\documentclass[review]{elsarticle}

\usepackage{lineno,hyperref, amsmath, amssymb, amsthm, arydshln, algorithmic, epstopdf, graphicx, amsfonts, multirow, xcolor, xfrac, subcaption}
\newtheorem{theorem}{Theorem}
\newtheorem{lemma}{Lemma}
\theoremstyle{remark}
\newtheorem{remark}{Remark}
\theoremstyle{definition}
\newtheorem{definition}{Definition}

\modulolinenumbers[5]

\journal{Linear Algebra and its Applications}

\begin{document}

\begin{frontmatter}

\title{Algebraic Linearizations of Matrix Polynomials}

\author{Eunice Y.~S.~Chan\fnref{western}}
\ead{echan295@uwo.ca}

\author{Robert M.~Corless\fnref{western}}
\ead{rcorless@uwo.ca}
\fntext[western]{Rotman Institute of Philosophy, Ontario Research Centre for Computer Algebra, School of Mathematical and Statistical Sciences, Department of Applied Mathematics, Western University}

\author{Laureano Gonzalez-Vega\fnref{cantabria}}
\fntext[cantabria]{Departamento de Matematicas, Estadistica y Computacion, Universidad de Cantabria}
\ead{laureano.gonzalez@unican.es}

\author{J.~Rafael Sendra\fnref{alcala}}
\fntext[alcala]{Research Group ASYNACS, Departamento de F\'{\i}sica y Matem\'aticas, Universidad de Alcal\'{a}}
\ead{rafael.sendra@uah.es}

\author{Juana Sendra\fnref{madrid}}
\fntext[madrid]{Matem\'atica Aplicada a las TIC, Universidad Polit{\'e}cnica de Madrid}
\ead{jsendra@etsist.upm.es}


\begin{abstract}
We show how to construct linearizations of matrix polynomials $z\mathbf{a}(z)\mathbf{d}_0 + \mathbf{c}_0$, $\mathbf{a}(z)\mathbf{b}(z)$, $\mathbf{a}(z) + \mathbf{b}(z)$ (when $\mathrm{deg}\left(\mathbf{b}(z)\right) < \mathrm{deg}\left(\mathbf{a}(z)\right)$), and $z\mathbf{a}(z)\mathbf{d}_0\mathbf{b}(z) + \mathbf{c_0}$ from linearizations of the component parts, $\mathbf{a}(z)$ and $\mathbf{b}(z)$. This allows the extension to matrix polynomials of a new companion matrix construction.
\end{abstract}


\end{frontmatter}


\section{Introduction}
Many applications require the computation or approximation of \textsl{polynomial eigenvalues}, that is, those $z \in \mathbb{C}$ for which the \textsl{matrix polynomial} $\mathbf{P}(z)$ (of degree at most $s$) $\in \mathbb{C}[z]^{r \times r}$ is singular. In other words, we search for $z$ such that $\operatorname{det}\mathbf{P}(z) = 0$. If $s = 1$, that is $\mathbf{P}(z) = z\mathbf{B} - \mathbf{A}$, where $\mathbf{A}$, $\mathbf{B} \in \mathbb{C}^{N \times N}$, where $N = r$, is degree 1 in $z$, \textsl{i.e.}~\textsl{linear}, then this is ``just" the generalized eigenvalue problem, which can be reliably solved numerically on many platforms using software developed over many decades by the efforts of many people. We do not here survey the state of the art of solving the generalized eigenvalue problem, \textsl{i.e.}~determining $z$ such that $\operatorname{det}\left(z\mathbf{B} - \mathbf{A}\right) = 0$ (provided the \textit{pencil} $(\mathbf{A}, \mathbf{B})$ is \textsl{regular}, \textsl{i.e.}~that $\operatorname{det}\left(z\mathbf{B} - \mathbf{A}\right) \not\equiv 0$). We do note that the so-called QZ iteration, which uses unitary transformations to simultaneously upper-triangularize $\mathbf{A}$ and $\mathbf{B}$ so that 
\begin{align}
	\operatorname{det}(z\mathbf{B} - \mathbf{A}) &= \operatorname{det}\mathbf{Q}\operatorname{det}\left(z\mathbf{B} - \mathbf{A}\right) \operatorname{det}\mathbf{Z} \\
	&= \operatorname{det}\left(z\mathbf{Q}\mathbf{B}\mathbf{Z} - \mathbf{Q}\mathbf{A}\mathbf{Z}\right) \\
	&= \operatorname{det}\left(z\mathbf{T_B} - \mathbf{T_A}\right)
\end{align}
allows its eigenvalues to be read off from the corresponding diagonal entries of $\mathbf{T_B}$ and $\mathbf{T_A}$, is by now very well-developed and reliable. Research continues into making the method even faster and more reliable especially as novel architectures are invented and especially for matrix structures that arise frequently in practice. But in this paper we simply take such methods as given: we regard a linear matrix polynomial as one that is effectively solved. Thus, our task becomes one of reducing a more general matrix polynomial eigenproblem to a ``mere" linear one. In this case, the dimension of the linear problem, $N$, is larger: $N \geq r \cdot s$ (remember the degree of $P(z)$ is at most $s$, and its dimension is $r$). This process is known as ``linearization", naturally enough, although we note that the resulting problem, even if it is called ``linear", is more properly considered as being of degree 2, once the unknown eigenvectors are considered: $z\mathbf{B}v = \mathbf{A}v$ is linear in the entries of $v$, and of $z$ by itself, but terms like $zv_1$, $zv_2$, etc appear, which are really of degree two, in the language of computational algebra. Indeed reduction of \textsl{any} system of polynomial equations (if there are only a finite number of solutions, a situation called ``being zero-dimensional'' in the literature) can always be ``reduced'' to a degree 2 system; this is known as the effective Nullstellensatz. Reduction to a generalized eigenproblem is a (very) practical concrete exhibition of this theorem.

Of course there are many practical details, that really matter. ``In theory, there's no difference between theory and practice; but in practice, there is." One huge item of practical importance is the commonly-undertaken reduction to upper Hessenberg form, prior to beginning the QZ iteration; this can be stably done in $\mathcal{O}(N^2)$ operations and greatly speeds up the iterations subsequently.

Other possibilities exist than linearization. Indeed there is much current research into what is called ``$\ell$-ification," \textsl{i.e.}~reduction of a matrix polynomial of degree $m\ell$ to a (larger) matrix polynomial of degree at most $\ell$ (having degree at most $\ell$ is also called ``having \textsl{grade} $\ell$")~\cite{dopico2018block}. But here we restrict ourselves to the case $\ell = 1$.

Surprisingly, there are still things to be said about this, in spite of many decades of work by many people. Of course, the proofs in this paper rely heavily on that work, especially that summarized in the classic \cite{gohberg2009matrix}. But still we will see some new elements, at least for a particular class of problems.

A useful introduction to the general area can be found in \cite[pages 263--281]{dennis2015princeton} and the references therein. Early history is discussed in \cite{rahman2002analytic}. Major recent works include \cite{mackey2006vector} and \cite{mackey2006structured}.

\section{The Basic Idea} \label{sec:3}

The basic idea of the algebraic linearizations described here was first discovered in the context of what are called ``Mandelbrot polynomials"~\cite{Chan2016, corless2013largest}. Mandelbrot polynomials are defined by $p_{0} = 0$ and $p_{n+1} = zp_{n}^{2} + 1$. Piers Lawrence found matrices $\mathbf{M}_{n}$, populated only by elements $0$ or $-1$, with $p_{n}(z) = \mathrm{det}(z\mathbf{I}- \mathbf{M}_{n})$. Naturally enough, these were called Mandelbrot matrices. We outline their construction below.

The first few $p_n$ are $p_0 = 0$, $p_1 = 1$, $p_2 = z+1$, and $p_3 = z^3 + 2z^2 + z + 1$. The idea is clearest going from $\mathbf{M}_3$ to $\mathbf{M}_4$; we will build up to that. Because the only root of $p_2$ is $z = -1$, clearly $\mathbf{M}_2 = \left[-1\right]$, a $1 \times 1$ matrix with eigenvalue $-1$. To make $\mathbf{M}_3$, glue two copies of $\mathbf{M}_2$ together to make
\begin{equation}
	\mathbf{M}_3 =
	\left[
		\begin{array}{ccc}
			\cdashline{1-1} \cdashline{3-3}
			\multicolumn{1}{:c:}{-1} & \phantom{-}0 & \multicolumn{1}{:c:}{\color{red}{-1}} \\
			\cdashline{1-3}
			\multicolumn{1}{:c}{\color{red}{-1}} & \multicolumn{1}{c:}{\phantom{-}0} & \phantom{-}0 \\
			\cdashline{3-3}
			\multicolumn{1}{:c}{\phantom{-}0} & \multicolumn{1}{c:}{\color{red}{-1}} & \multicolumn{1}{:c:}{-1}\\
			\cdashline{1-3}
		\end{array}
	\right]
\end{equation}
which one can directly verify has
\begin{align}
	\operatorname{det}\left(z\mathbf{I} - \mathbf{M}_3\right) &= \operatorname{det}
	\left(
		\begin{array}{ccc}
			z + 1 & 0 & 1 \\
			1 & z & 0 \\
			& 1 & z+1
		\end{array}
	\right) \\
	&= (z + 1) \operatorname{det}
	\left(
		\begin{array}{cc}
			z  & 0 \\
			1 & z + 1
		\end{array}
	\right)
	+ 1 \cdot \operatorname{det}
	\left(
		\begin{array}{cc}
			1 & z \\
			& 1
		\end{array}
	\right) \\
	&= z(z+1)^2 + 1 \>, \ \text{as desired.}
\end{align}
To make $\mathbf{M}_4$ we glue two copies of $\mathbf{M}_3$ together:
\begin{equation}
	\left[
		\begin{array}{ccccccc}
			\cdashline{1-3} \cdashline{7-7}
			\multicolumn{1}{:c}{-1} & \phantom{-}0 & \multicolumn{1}{c:}{-1} & & & & \multicolumn{1}{:c:}{\color{red}{-1}} \\
			\cdashline{7-7}
			\multicolumn{1}{:c}{-1} & \phantom{-}0 & \multicolumn{1}{c:}{\phantom{-}0} & & & & \\
			\multicolumn{1}{:c}{ }& -1 & \multicolumn{1}{c:}{-1} & & & & \\
			\cdashline{1-4}
			& & \multicolumn{1}{:c}{\color{red}{-1}} & \multicolumn{1}{c:}{\phantom{-}0} & & & \\
			\cdashline{5-7}
			& & \multicolumn{1}{:c}{ } & \multicolumn{1}{c:}{\color{red}{-1}} & -1 & \phantom{-}0 & \multicolumn{1}{c:}{-1} \\
			\cdashline{3-4}
			& & & & \multicolumn{1}{:c}{-1} & \phantom{-}0 & \multicolumn{1}{c:}{\phantom{-}0} \\
			& & & & \multicolumn{1}{:c}{ } & -1 & \multicolumn{1}{c:}{-1}\\
			\cdashline{5-7}
		\end{array}
	\right]
\end{equation}
and at this level the ``glue" and the ``copies" are more distingushable. The upper Hessenberg nature of the matrix is also visible. To prove $p_4 = \operatorname{det}\left(z\mathbf{I} - \mathbf{M}_4\right)$ we use Knuth's idea: the determinant is linear in the first row:
\begin{align}
	&\operatorname{det}\left(z\mathbf{I} - \mathbf{M}_4\right) \\
	&= \operatorname{det}\left(
		\begin{array}{c:c:c}
			z\mathbf{I} - \mathbf{M}_{3} &
			\begin{array}{c}
				\ \\
				\
			\end{array}
			& \\
			\hdashline
			\begin{array}{cr}
				\quad & 1
			\end{array}
			& z & \\
			\hdashline
			& 
			\begin{array}{c}
				1 \\
				\
			\end{array}
			&
			z\mathbf{I} - \mathbf{M}_{3}
		\end{array}
	\right) +
	\operatorname{det} \left(
	\begin{array}{ccccc}
		0 & 0 & 0 & 0 & 1 \\
		1 & \\
		& 1 & & & \\
		& & \ddots & & \makebox(2,0){\rule[5.5ex]{1pt}{3.5\normalbaselineskip}} \\
		& & & 1
	\end{array}
	\right) \\
	&= zp_3^2 + 1 \cdot \operatorname{det}\left(
	\begin{array}{cccc}
		1 & & \makebox(0,2){\rule[1ex]{37pt}{1pt}}& \\
		& 1 & & \\
		& & \ddots & \makebox(2, 0){\rule[9ex]{1pt}{45pt}}\\
		& & & 1
	\end{array}
	\right) \\
	&= zp_3^2 + 1 \>, \ \text{as desired.}
\end{align}
This gives the idea. The generalization will be Theorem \ref{thm:4} in the next section.

For more on Mandelbrot matrices, see \cite{corless2013largest}, \cite{Chan2016}, and \cite{Chan2017}. They and their generalizations have some interesting properties. For now, note that \cite{Chan2017} generalized the construction to finding a companion for the scalar polynomial $c = z\mathbf{ab} + \mathbf{c}_0$ given upper Hessenberg companions for $\mathbf{a}$ and $\mathbf{b}$. It is that generalization that we turn into a linearization in the next section.

\section{The Main Theorems}
Theorem \ref{thm:1} shows how to linearize 
\begin{equation}
	\mathbf{e}_{1}(z) = z\mathbf{d}_{0}\mathbf{a}(z) + \mathbf{c}_{0}
	\label{eqn:e1}
\end{equation}
and 
\begin{equation}
	\mathbf{e}_{2}(z) = z\mathbf{a}(z)\mathbf{d}_{0} + \mathbf{c}_{0} \>,
	\label{eqn:e2}
\end{equation}
where $\mathbf{e}_{1}(z), \mathbf{e}_{2}(z) \in \mathbb{C}^{r \times r}$, once linearization for $\mathbf{a}(z)$ is available.
\begin{equation}
	\Lambda\left(\mathbf{a}(z)\right) := \left\{z \mid \mathrm{det}\left(\mathbf{a}(z)\right) = 0\right\}
\end{equation}
is the spectrum of the matrix polynomial $\mathbf{a}(z) \in \mathbb{C}^{r \times r}$. These $z$ are the polynomial eigenvalues of $\mathbf{a}(z)$.
\begin{theorem} \label{thm:1}
Consider $\mathbf{e}_{1}(z)$ and $\mathbf{e}_{2}$ as in equations \eqref{eqn:e1} and \eqref{eqn:e2}, respectively. Suppose $\mathbf{a}(z) \in \mathbb{C}[z]^{r\times r}$ is of degree $s \geq 1$, $\mathbf{c}_{0}$ and $\mathbf{d}_0 \in \mathbb{C}^{r \times r}$,  and that $\mathbf{a}(z)$ has the regular linearization pencil $(\mathbf{D_A}, \mathbf{A})$ with $\operatorname{det} \mathbf{a}(z) = \operatorname{det} \left(z\mathbf{D_A} - \mathbf{A}\right)$ and $z\mathbf{D_A} - \mathbf{A}$ invertible except when $z \in \Lambda(\mathbf{a})$ which is a discrete set. Moreover suppose that we have the resolvent form
\begin{equation}
	\mathbf{a}^{-1}(z) = \mathbf{X}_\mathbf{A}\left(z\mathbf{D_A} - \mathbf{A}\right)^{-1}\mathbf{Y}_\mathbf{A} \qquad z \in \mathbb{C}\notin \Lambda(\mathbf{a})
\end{equation}
and $\mathbf{X}_\mathbf{A} \in \mathbb{C}^{r\times rs}$ and $\mathbf{Y}_{\mathbf{A}} \in \mathbb{C}^{rs \times r}$ are known. Then if
\begin{equation}
	\mathbf{E}_1 = \left[
	\begin{array}{cc}
		\mathbf{0} & \mathbf{c}_0\mathbf{X}_\mathbf{A} \\
		-\mathbf{Y}_\mathbf{A} & \mathbf{A}
	\end{array}
	\right] \>, \quad
	\mathbf{D}_{\mathbf{E}_1} = \left[
	\begin{array}{cc}
		\mathbf{d}_0 & \\
		& \mathbf{D_A}
	\end{array}
	\right]
\end{equation}
and
\begin{equation}
	\mathbf{E}_2 = \left[
	\begin{array}{cc}
		\mathbf{A} & \mathbf{Y}_\mathbf{A}\mathbf{c}_0 \\
		-\mathbf{X}_\mathbf{A} & \mathbf{0}
	\end{array}
	\right] \>, \quad \mathbf{D}_{\mathbf{E}_2} = \left[
	\begin{array}{cc}
		\mathbf{D_A} & \\
		& \mathbf{d}_0
	\end{array}
	\right]
\end{equation}
then $\operatorname{det}\left(z\mathbf{D}_{\mathbf{E}_1} - \mathbf{E}_1\right) = \operatorname{det} \mathbf{e}_1(z)$ where $\mathbf{e}_1(z) = z\mathbf{d}_0 \mathbf{a}(z) + \mathbf{c}_0$ and \\ $\operatorname{det}\left(z\mathbf{D}_{\mathbf{E}_2} - \mathbf{E}_2\right) = \operatorname{det} \mathbf{e}_2(z)$ where $\mathbf{e}_2(z) = z\mathbf{a}(z)\mathbf{d}_0 + \mathbf{c}_0$. Moreoever
\begin{equation}
	\left[
		\begin{array}{cc}
			\mathbf{0} & -\mathbf{X}_\mathbf{A}
		\end{array}
	\right]
	\left(z\mathbf{D}_{\mathbf{E}_1} - \mathbf{E}_1\right)^{-1}
	\left[
		\begin{array}{c}
			\mathbf{I} \\
			\mathbf{0}
		\end{array}
	\right] = \mathbf{e}_1^{-1}(z)
\end{equation}
and
\begin{equation}
	\left[
		\begin{array}{cc}
			\mathbf{0} & \mathbf{I}
		\end{array}
	\right]
	\left(z\mathbf{D}_{\mathbf{E}_2} - \mathbf{E}_2\right)^{-1}
	\left[
		\begin{array}{c}
			-\mathbf{Y}_\mathbf{A} \\
			\mathbf{0}
		\end{array}
	\right] = \mathbf{e}_2^{-1}(z)
\end{equation}
give resolvent forms for the larger systems.
\begin{proof}
We use the Schur factoring~\cite[Chapter 12]{hogben2006handbook}:
\begin{align}
	&z\mathbf{D}_{\mathbf{E}_1} - \mathbf{E}_1\\ 
	&=
	\left[
		\begin{array}{cc}
			z\mathbf{d}_0 & -\mathbf{c}_0\mathbf{X}_\mathbf{A} \\
			\mathbf{Y}_\mathbf{A} & z\mathbf{D_A} - \mathbf{A}
		\end{array}
	\right] \\
	&=
	\left[
		\begin{array}{cc}
			\mathbf{I} & -\mathbf{c}_0\mathbf{X}_\mathbf{A}\left(z\mathbf{D_A} - \mathbf{A}\right)^{-1} \\
			\mathbf{0} & \mathbf{I}
		\end{array}
	\right]
	\left[
		\begin{array}{cc}
			 \mathbf{S_A} & \mathbf{0} \\
			\mathbf{Y}_\mathbf{A} & z\mathbf{D_A} - \mathbf{A}
		\end{array}
	\right] \\
	&= \left[
		\begin{array}{cc}
			\mathbf{I} & -\mathbf{c}_0\mathbf{X}_\mathbf{A}\left(z\mathbf{D_A} - \mathbf{A}\right)^{-1} \\
			\mathbf{0} & \mathbf{I}
		\end{array}
	\right]
	\left[
		\begin{array}{cc}
			z\mathbf{d}_0 + \mathbf{c}_0\mathbf{a}^{-1}(z) & \mathbf{0} \\
			\mathbf{Y}_\mathbf{A} & z\mathbf{D_A} - \mathbf{A}
		\end{array}
	\right] \\
	&= \mathbf{P}_1\mathbf{P}_2 \>,
\end{align}
where the Schur complement $\mathbf{S_A} = z\mathbf{d}_0 + \mathbf{c}_0\mathbf{X}_\mathbf{A}\left(z\mathbf{D_A} - \mathbf{A}\right)^{-1}\mathbf{Y}_\mathbf{A}$. Thus
\begin{align}
	\operatorname{det}(z\mathbf{D}_{\mathbf{E}_1} - \mathbf{E}_1) &= \mathrm{det}(\mathbf{P}_{1})\mathrm{det}(\mathbf{P}_{2}) \\
	&= \operatorname{det}\left(z\mathbf{d}_0 + \mathbf{c}_0\mathbf{a}^{-1}(z)\right)\operatorname{det}\left(z\mathbf{D_A} - \mathbf{A}\right)\\
	&=\operatorname{det}\left(z\mathbf{d}_0 + \mathbf{c}_0\mathbf{a}^{-1}(z)\right)\operatorname{det}\mathbf{a}(z) \\
	&= \operatorname{det}\left(z\mathbf{d}_0\mathbf{a}(z) + \mathbf{c}_0 \right)\\
	&= \operatorname{det}\left(\mathbf{e}_1(z)\right) \>,
\end{align}
as desired. Moreover, $\left(z\mathbf{D}_{\mathbf{E}_1} - \mathbf{E}_1\right)^{-1} = \mathbf{P}^{-1}_2\mathbf{P}^{-1}_1$. Let
\begin{align}
	\mathbf{Q}_{a} &= 
	\left[
		\begin{array}{cc}
			\mathbf{I} & \\
			& \left(z\mathbf{D_A} - \mathbf{A}\right)^{-1}
		\end{array}
	\right]
	\left[
		\begin{array}{cc}
			\mathbf{I} & \mathbf{0} \\
			-\mathbf{Y}_\mathbf{A} & \mathbf{I}
		\end{array}
	\right]
	\left[
		\begin{array}{cc}
			\left(\mathbf{e}_1\mathbf{a}^{-1}\right)^{-1} & \\
			& \mathbf{I}
		\end{array}
	\right] \\
	&=
	\left[
		\begin{array}{cc}
			\mathbf{a}(z)\mathbf{e}_1^{-1}(z) & \mathbf{0} \\
			-\left(z\mathbf{D_A} - \mathbf{A}\right)^{-1}\mathbf{Y}_\mathbf{A} \mathbf{a}(z)\mathbf{e}_1^{-1}(z) & \left(z\mathbf{D_A} - \mathbf{A}\right)^{-1}
		\end{array}
	\right]
\end{align}
and
\begin{equation}
	\mathbf{Q}_{b} = 
	\left[
		\begin{array}{cc}
			\mathbf{I} & \mathbf{c}_0\mathbf{X}_\mathbf{A}\left(z\mathbf{D_A} - \mathbf{A}\right)^{-1} \\
			\mathbf{0} & \mathbf{I}
		\end{array}
	\right] \>.
\end{equation}
Then,
\begin{align}
	\mathbf{P}^{-1}_2 &= \mathbf{Q}_{a}\mathbf{Q}_{b} \\
	&= 
	\left[
		\begin{array}{cc}
			\mathbf{a}(z)\mathbf{e}_1^{-1}(z) & \mathbf{Q}_{c} \\
			-\left(z\mathbf{D_A} - \mathbf{A}\right)^{-1}\mathbf{Y}_\mathbf{A} \mathbf{a}(z)\mathbf{e}_1^{-1}(z) & \mathbf{Q}_{d}
		\end{array}
	\right] \>,
\end{align}
where
\begin{equation}
	\mathbf{Q}_{c} = \mathbf{a}(z)\mathbf{e}_1^{-1}(z)\mathbf{c}_0\mathbf{X}_\mathbf{A}\left(z\mathbf{D_A} - \mathbf{A}\right)^{-1}
\end{equation}
and
\begin{equation}
	\mathbf{Q}_{d} = \left(z\mathbf{D_A} - \mathbf{A}\right)^{-1} - \left(z\mathbf{D_A} - \mathbf{A}\right)^{-1}\mathbf{Y}_\mathbf{A} \mathbf{a}\mathbf{e}_1^{-1}\mathbf{c}_0\mathbf{X}_\mathbf{A}\left(z\mathbf{D_A} - \mathbf{A}\right)^{-1} \>,
\end{equation}
so
\begin{align}
	\left[
		\begin{array}{cc}
			\mathbf{0} & - \mathbf{X}_\mathbf{A}
		\end{array}
	\right]
	\left(z\mathbf{D}_{\mathbf{E}_1} - \mathbf{E}_1\right)^{-1}
	\left[
		\begin{array}{c}
			\mathbf{I} \\
			\mathbf{0}
		\end{array}
	\right]
	&= \mathbf{X}_\mathbf{A}\left(z\mathbf{D_A}-\mathbf{A}\right)^{-1} \mathbf{Y}_\mathbf{A} \mathbf{a}(z)\mathbf{e}_1^{-1}(z) \\
	&= \mathbf{a}^{-1}(z)\mathbf{a}(z)\mathbf{e}^{-1}_1(z)\\
	&= \mathbf{e}_1^{-1}(z)
\end{align}
as claimed.

Similarly,
\begin{align}
	z\mathbf{D}_{\mathbf{E}_2} - \mathbf{E}_2 &= 
	\left[
		\begin{array}{cc}
			z\mathbf{D_A} - \mathbf{A} & -\mathbf{Y}_\mathbf{A}\mathbf{c}_0 \\
			\mathbf{X}_\mathbf{A} & z\mathbf{d}_0
		\end{array}
	\right]\\ 
	&=\mathbf{Q}_{e}\mathbf{Q}_{f} \>,
\end{align}
where
\begin{equation}
	\mathbf{Q}_{e} = \left[
		\begin{array}{cc}
			z\mathbf{D_A} - \mathbf{A} & \mathbf{0} \\
			\mathbf{X}_\mathbf{A} & z\mathbf{d}_0 + \mathbf{X}_{\mathbf{A}}\left(z\mathbf{D_A} - \mathbf{A}\right)^{-1}\mathbf{Y}_\mathbf{A}\mathbf{c}_0
		\end{array}
	\right] 
\end{equation}
and
\begin{equation}
	\mathbf{Q}_{f} = 
	\left[
		\begin{array}{cc}
			\mathbf{I} & -\left(z\mathbf{D_A} - \mathbf{A}\right)^{-1}\mathbf{Y}_\mathbf{A}\mathbf{c}_0 \\
			\mathbf{0} & \mathbf{I}
		\end{array}
	\right] \>,
\end{equation}
so
\begin{align}
	\operatorname{det}\left(z\mathbf{D}_{\mathbf{E}_2} - \mathbf{E}_2\right) & = \operatorname{det}\left(z\mathbf{D_A} - \mathbf{A}\right)\operatorname{det}\left(z\mathbf{d}_0 + \mathbf{a}^{-1}(z) \mathbf{c}_0\right) \\
	&= \operatorname{det} \mathbf{a} \operatorname{det}\left(z\mathbf{d}_0 + \mathbf{a}^{-1}\mathbf{c}_0\right) \\
	&= \operatorname{det}\left(z \mathbf{a}(z) \mathbf{d}_0 + \mathbf{c}_0\right) \\
	&= \operatorname{det} \mathbf{e}_2(z)
\end{align}
as claimed.

Moreover
\begin{equation}
	\left(z\mathbf{D}_{\mathbf{E}_2} - \mathbf{E}_2\right)^{-1} = \mathbf{Q}_{g} \mathbf{Q}_{h} \>,
\end{equation}
where
\begin{align}
	\mathbf{Q}_{g} &= 
	\left[
		\begin{array}{cc}
			\mathbf{I} & \left(z\mathbf{D_A} - \mathbf{A}\right)^{-1}\mathbf{Y}_\mathbf{A}\mathbf{c}_0 \\
			\mathbf{0} & \mathbf{I}
		\end{array}
	\right]
	\left[
		\begin{array}{cc}
			\mathbf{I} & \\
			& \mathbf{e}_2^{-1}(z)\mathbf{a}(z)
		\end{array}
	\right] \\
	&= \left[
		\begin{array}{cc}
			\mathbf{I} & \left(z\mathbf{D_A} - \mathbf{A} \right)^{-1}\mathbf{Y}_\mathbf{A} \mathbf{c}_0\mathbf{e}_2^{-1}\mathbf{a} \\
			 \mathbf{0} & \mathbf{e}_2^{-1}\mathbf{a}
		\end{array}
	\right]
\end{align}
and
\begin{align}
	\mathbf{Q}_{h} &=
	\left[
		\begin{array}{cc}
			\mathbf{I} & \mathbf{0} \\
			-\mathbf{X}_\mathbf{A} & \mathbf{I}
		\end{array}
	\right]
	\left[
		\begin{array}{cc}
			\left(z\mathbf{D_A} - \mathbf{A}\right)^{-1} & \\
			& \mathbf{I}
		\end{array}
	\right] \\
	&= 
	\left[
		\begin{array}{cc}
			\left(z\mathbf{D_A} - \mathbf{A}\right)^{-1} & \mathbf{0} \\
			-\mathbf{X}_\mathbf{A}\left(z\mathbf{D_A} - \mathbf{A}\right)^{-1} & \mathbf{I}
		\end{array}
	\right] \>,
\end{align}
which results in
\begin{equation}
	\left[
		\begin{array}{cc}
			\mathbf{Q}_{i} & \left(z\mathbf{D_A} - \mathbf{A}\right)^{-1}\mathbf{Y}_\mathbf{A}\mathbf{c}_0\mathbf{e}_2^{-1}\mathbf{a} \\
			-\mathbf{e}_2^{-1}\mathbf{a}\mathbf{X}_\mathbf{A}\left(z\mathbf{D_A} - \mathbf{A}\right)^{-1} & \mathbf{e}_2^{-1}\mathbf{a}
		\end{array}
	\right] \>,
\end{equation}
where
\begin{equation}
	\mathbf{Q}_{i} =
	\left(z\mathbf{D_A} - \mathbf{A}\right)^{-1} - \left(z\mathbf{D_A} - \mathbf{A}\right)^{-1}\mathbf{Y}_\mathbf{A}\mathbf{c}_0\mathbf{e}_2^{-1}\mathbf{a}\mathbf{X}_\mathbf{A}\left(z\mathbf{D_A} - \mathbf{A}\right)^{-1} \>.
\end{equation}
Therefore,
\begin{align}
	\left[
		\begin{array}{cc}
			\mathbf{0} & \mathbf{I}
		\end{array}
	\right]
	\left(z\mathbf{D}_2 - \mathbf{E}_2\right)^{-1}
	\left[
		\begin{array}{c}
			-\mathbf{Y}_\mathbf{A} \\
			\mathbf{0}
		\end{array}
	\right]
	&= \mathbf{e}_2^{-1}\mathbf{a}\mathbf{X}_\mathbf{A}\left(z\mathbf{D_A} - \mathbf{A}\right)^{-1}\mathbf{Y}_\mathbf{A} \\
	&= \mathbf{e}_2^{-1}(z)
\end{align}
as claimed.
\end{proof}
\end{theorem} 

Theorem \ref{thm:2} shows how to linearize a product $\mathbf{a}(z)\mathbf{b}(z)$ given linearizations of each of $\mathbf{a}(z)$ and $\mathbf{b}(z)$.
\begin{theorem}\label{thm:2}
Suppose $\mathbf{a}(z)$, $\mathbf{D_A}$, $\mathbf{A}$, $\mathbf{X}_\mathbf{A}$, and $\mathbf{Y}_\mathbf{A}$ are as in Theorem \ref{thm:1}, and suppose similarly that $\mathbf{b}(z) \in \mathbb{C}[z]^{r \times r}$ is of degree $t \geq 1$, has the regular linearization pencil $\left(\mathbf{D_B}, \mathbf{B}\right)$ with $\operatorname{det}\mathbf{b}(z) = \operatorname{det}\left(z\mathbf{D_B} - \mathbf{B}\right)$ and resolvent
\begin{equation}
	\mathbf{b}^{-1}(z) = \mathbf{X}_\mathbf{B}\left(z\mathbf{D_B} - \mathbf{B}\right)^{-1}\mathbf{Y}_\mathbf{B} \qquad \text{for } z \in \mathbb{C} \notin \Lambda(\mathbf{b})
\end{equation}
\end{theorem}
\noindent
Then if we define 
\begin{equation}
	\mathbf{F}_1 = \left[
		\begin{array}{cc}
			\mathbf{A} & \mathbf{0} \\
			\mathbf{Y}_\mathbf{B}\mathbf{X}_\mathbf{A} & \mathbf{B}
		\end{array}
	\right]
	\quad \text{and} \quad
	\mathbf{D}_{\mathbf{F}_1} =
	\left[
		\begin{array}{cc}
			\mathbf{D_A} & \mathbf{0} \\
			\mathbf{0} & \mathbf{D_B}
		\end{array}
	\right]
	\label{eqn:F1}
\end{equation}
or similarly
\begin{equation}
	\mathbf{F}_2 = 
	\left[
		\begin{array}{cc}
			\mathbf{B} & \mathbf{Y}_\mathbf{B}\mathbf{X}_\mathbf{A} \\
			\mathbf{0} & \mathbf{A}
		\end{array}
	\right]
	\quad \text{and} \quad
	\mathbf{D}_{\mathbf{F}_2} =
	\left[
		\begin{array}{cc}
			\mathbf{D_B} & \mathbf{0} \\
			\mathbf{0} & \mathbf{D_A}
		\end{array}
	\right] \>,
	\label{eqn:F2}
\end{equation}
then $z\mathbf{D}_{\mathbf{F}_{1}} - \mathbf{F}_{1}$ and $z\mathbf{D}_{\mathbf{F}_{2}} - \mathbf{F}_{2}$ are linearizations for $\mathbf{a}(z)\mathbf{b}(z)$.

\begin{proof}
Consider $\mathbf{F}_{1}$, $\mathbf{D}_{\mathbf{F}_{1}}$, $\mathbf{F}_{2}$, and $\mathbf{D}_{\mathbf{F}_{2}}$ as in equations \eqref{eqn:F1} and \eqref{eqn:F2} shown above. Clearly $\mathbf{A}$ and $\mathbf{B}$ can be exchanged in either factor to get new but related constructions. Then
\begin{align}
	\operatorname{det}\left(z\mathbf{D}_{\mathbf{F}_1} - \mathbf{F}_1\right) &= \operatorname{det}
	\left(
		\begin{array}{cc}
			z\mathbf{D_A} - \mathbf{A} & \mathbf{0} \\
			-\mathbf{Y}_\mathbf{B}\mathbf{X}_\mathbf{A} & z\mathbf{D_B} - \mathbf{B}
		\end{array}
	\right) \\
	&= \operatorname{det}\mathbf{a}(z) \operatorname{det}\mathbf{b}(z) \\
	&= \operatorname{det} \mathbf{a}(z)\mathbf{b}(z)
\end{align}
and moreover
\begin{align}
	&\left(z\mathbf{D}_{\mathbf{F}_1} - \mathbf{F}_1\right)^{-1} = \\
	&\left[
		\begin{array}{cc}
			\left(z\mathbf{D_A} - \mathbf{A}\right)^{-1} & \mathbf{0} \\
			\left(z\mathbf{D_B} - \mathbf{B}\right)^{-1}\mathbf{Y}_\mathbf{B}\mathbf{X}_\mathbf{A}\left(z\mathbf{D_A} - \mathbf{A}\right)^{-1} & \left(z\mathbf{D_B} - \mathbf{B}\right)^{-1}
		\end{array}
	\right]
\end{align}
so
\begin{equation}
	\left[
		\begin{array}{cc}
			\mathbf{0} & \mathbf{X}_\mathbf{B}
		\end{array}
	\right]
	\left(z\mathbf{D}_{\mathbf{F}_1} - \mathbf{F}_1\right)^{-1}
	\left[
		\begin{array}{c}
			\mathbf{Y}_\mathbf{A} \\
			\mathbf{0}
		\end{array}
	\right]
	= \mathbf{b}^{-1}(z) \mathbf{a}^{-1}(z) \>.
\end{equation}

Reversing $\mathbf{A}$ and $\mathbf{B}$ in $\mathbf{F}_1$ gives instead $\mathbf{a}^{-1}(z)\mathbf{b}^{-1}(z)$. Similarly
\begin{equation}
	z\mathbf{D}_{\mathbf{F}_2} - \mathbf{F}_2 =
	\left[
		\begin{array}{cc}
			z\mathbf{D_B} - \mathbf{B} & -\mathbf{Y}_{\mathbf{B}}\mathbf{X}_\mathbf{A} \\
			\mathbf{0} & z\mathbf{D_A} - \mathbf{A}
		\end{array}
	\right] \>
\end{equation}
so again $\operatorname{det}\left(z\mathbf{D}_{\mathbf{F}_2} - \mathbf{F}_2\right) = \operatorname{det}\mathbf{b}(z)\operatorname{det}\mathbf{a}(z) = \operatorname{det}\left(\mathbf{b}(z)\mathbf{a}(z)\right)$. Moreover
\begin{align}
	&\left(z\mathbf{D}_{\mathbf{F}_2} - \mathbf{F}_2 \right)^{-1} =\\
	&\left[
		\begin{array}{cc}
			\left(z\mathbf{D_B} - \mathbf{B}\right)^{-1} & \left(z\mathbf{D_B} - \mathbf{B}\right)^{-1}\mathbf{Y}_\mathbf{B}\mathbf{X}_\mathbf{A}\left(z\mathbf{D_A} - \mathbf{A}\right)^{-1} \\
			\mathbf{0} & \left(z\mathbf{D_A} - \mathbf{A}\right)^{-1}
		\end{array}
	\right]
\end{align}
so
\begin{equation}
	\left[
		\begin{array}{cc}
			\mathbf{X}_\mathbf{B} & \mathbf{0}
		\end{array}
	\right]
	\left(z\mathbf{D}_{\mathbf{F}_2} - \mathbf{F}_2\right)^{-1}
	\left[
		\begin{array}{c}
			\mathbf{0} \\
			\mathbf{Y}_\mathbf{A}
		\end{array}
	\right]
	= \mathbf{b}^{-1}(z)\mathbf{a}^{-1}(z) \>.
\end{equation}
\end{proof}

\begin{remark}
Theorem \ref{thm:2} is just Theorem 3.2 from \cite[p.~85]{gohberg2009matrix} with two minor modifications: non-monic $\mathbf{a}(z)$ is covered here, and we will use $\mathbf{F}_2$ to give a block upper Hessenberg matrix whereas they use $\mathbf{F}_1$. That seems paradoxical because $\mathbf{F}_1$ looks more likely to generate block upper Hessenberg matrices, but when used recursively the lower left triangle remains empty when $\mathbf{X}_\mathbf{A}$ and $\mathbf{Y}_\mathbf{B}$ are $e_s^T \otimes I_r$ and $f_1 \otimes I_r$. We will need the upper right block for the constant coefficient added.
\end{remark}

Theorems \ref{thm:3} and \ref{thm:4} show how to linearize $\mathbf{a}(z) + \mathbf{c}(z)$ if $\mathrm{deg}(\mathbf{c}(z)) < \mathrm{deg}(\mathbf{a}(z))$. Theorem \ref{thm:3} considers the monic case for $\mathbf{a}(z)$, and Theorem \ref{thm:4} relaxes this restriction.
\begin{theorem}[monic case]\label{thm:3}
Suppose $\mathbf{a}(z) = z^s + \boldsymbol{\alpha}_{s-1}z^{s-1} + \cdots + \boldsymbol{\alpha}_0$ and each $\boldsymbol{\alpha}_{k} \in \mathbb{C}^{r \times r}$, and that we have a block upper Hessenberg linearization $\mathbf{A}$ of $\mathbf{a}(z)$ with standard triple $\mathbf{X}_\mathbf{A}$, $\mathbf{A}$, $\mathbf{Y}_\mathbf{A}$ which means among other things that
\begin{equation}
	\mathbf{X}_\mathbf{A}\left(z\mathbf{I}_{sr} - \mathbf{A}\right)^{-1}\mathbf{Y}_\mathbf{A} = \mathbf{a}^{-1}(z) \>.
\end{equation}
Then if $\mathbf{c}(z) = \mathbf{c}_{s-1}z^{s-1} + \cdots + \mathbf{c}_1z + \mathbf{c}_0$, with each $\mathbf{c}_i \in \mathbb{C}^{r \times r}$, is of degree at most $s - 1$, then 
\begin{equation}
	\mathbf{G} = \mathbf{A} - \sum_{k = 0}^{s - 1} \mathbf{A}^{k}\mathbf{Y}_\mathbf{A}\mathbf{c}_k\mathbf{X}_\mathbf{A}
\end{equation}
is a block upper Hessenberg linearization of $\mathbf{a}(z) + \mathbf{c}(z)$, with 
\begin{equation}
	\mathbf{X}_\mathbf{A}\left(z\mathbf{I} - \mathbf{G}\right)^{-1}\mathbf{Y}_\mathbf{A} = \left(\mathbf{a}(z) + \mathbf{c}(z)\right)^{-1} \>.
\end{equation} 
\end{theorem}

\begin{proof}
Using the properties of a standard triple \cite[see Proposition 2.1 (i), p 53]{gohberg2009matrix} the matrix
\begin{equation}
	\mathbf{V} = 
	\left[
		\begin{array}{ccccc}
			\mathbf{Y}_\mathbf{A} & \mathbf{A}\mathbf{Y}_\mathbf{A} & \mathbf{A}^2\mathbf{Y}_\mathbf{A} & \cdots & \mathbf{A}^{s-1}\mathbf{Y}_\mathbf{A}
		\end{array}
	\right]
\end{equation}
is nonsingular. Put $\mathbf{V}_k = \mathbf{A}^{k-1}\mathbf{Y}_\mathbf{A}$ for $1 \leq k \leq s$. Note each $\mathbf{V}_k$ is $sr$ by $r$. Then direct computation shows
\begin{equation}
	\mathbf{A}
	\left[
		\begin{array}{cccc}
			\mathbf{V}_1 & \mathbf{V}_2 & \ldots & \mathbf{V}_s 
		\end{array}
	\right] =
	\left[
		\begin{array}{ccccc}
			\mathbf{V}_2 & \mathbf{V}_3 & \cdots & \mathbf{V}_s & \mathbf{A}^{s}\mathbf{Y}_\mathbf{A}
		\end{array}
	\right]
\end{equation}
By part (iii) of the previously mentioned proposition,
\begin{equation}
	\mathbf{A}^{s}\mathbf{Y}_\mathbf{A} = -\sum_{k = 1}^{s} \mathbf{A}^{k-1}\mathbf{Y}_{\mathbf{A}}\boldsymbol{\alpha}_{k-1} \>,
\end{equation}
meaning that the given matrix polynomial is ``solved" by its linearization times $\mathbf{Y}_\mathbf{A}$ (a generalization of the Cayley-Hamilton theorem).

Thus
\begin{equation}
	\mathbf{A}\mathbf{V} = \mathbf{V}
	\left[
		\begin{array}{ccccc}
			\mathbf{0} & \mathbf{0} & \cdots & \mathbf{0} & -\boldsymbol{\alpha}_0 \\
			\mathbf{I} & \mathbf{0} & & & -\boldsymbol{\alpha}_1 \\
			& \mathbf{I} & \ddots & & \vdots \\
			& & \ddots & \mathbf{0} & -\boldsymbol{\alpha}_{s-2} \\
			& & & \mathbf{I} & -\boldsymbol{\alpha}_{s-1}
		\end{array}
	\right] = \mathbf{V}\mathbf{C}_2
\end{equation}
where $\mathbf{C}_2$ is the familiar ``second companion linearization", making explicit the similarity $\mathbf{A} = \mathbf{V}\mathbf{C}_2\mathbf{V}^{-1}$. Quite clearly the second companion linearization of $\mathbf{a}(z) + \mathbf{c}(z)$ is
\begin{equation}
	\left[
		\begin{array}{cccc}
			\mathbf{0} & & & -\left(\boldsymbol{\alpha}_0 + \mathbf{c}_0\right) \\
			\mathbf{I} & \mathbf{0} & & -\left(\boldsymbol{\alpha}_1 + \mathbf{c}_1\right) \\
			& \mathbf{I} & & \vdots \\
			& & \ddots & \\
			& & & \mathbf{I} - \left(\boldsymbol{\alpha}_{s-1} + \mathbf{c}_{s-1}\right)
		\end{array}
	\right]
\end{equation}
and we look for a matrix $\mathbf{W}$ such that $\mathbf{A} + \mathbf{W}$ linearizes $\mathbf{a}(z) + \mathbf{c}(z)$. Now
\begin{equation}
	\mathbf{W}\mathbf{V} = \mathbf{V}\Delta \mathbf{C}_2 = \mathbf{V}
	\left[
		\begin{array}{ccccc}
			\mathbf{0} & \mathbf{0} & \cdots & \mathbf{0} & -\mathbf{c}_0 \\
			\mathbf{0} & \cdots & \cdots & \mathbf{0} & -\mathbf{c}_1 \\
			\vdots & \ddots & & \vdots & \vdots \\
			\mathbf{0} & \cdots & \cdots & \mathbf{0} & -\mathbf{c}_{s-1} 
		\end{array}
	\right]
\end{equation}
implies
\begin{align}
	\mathbf{W} &= \mathbf{V}
	\left[
		\begin{array}{cccc}
			\mathbf{0} & \cdots & \mathbf{0} & -\mathbf{c}_0 \\
			\vdots & & \vdots & -\mathbf{c}_1 \\
			\vdots & & \vdots & \vdots \\
			\mathbf{0} & \cdots & \mathbf{0} & -\mathbf{c}_{s-1}
		\end{array}
	\right]
	\mathbf{V}^{-1} \\
	&=
	\left[
		\begin{array}{cccc}
			\mathbf{0} & \cdots & \mathbf{0} & \multirow{3}{*}{$-\sum_{k=1}^{s}\mathbf{V}_k\mathbf{c}_{k-1}$} \\
			\vdots & & \vdots &  \\
			\mathbf{0} & \cdots & \mathbf{0} &
		\end{array}
	\right]
	\mathbf{V}^{-1} \\
	&= -\sum_{k=1}^{s}\mathbf{V}_k \mathbf{c}_{k-1}\mathbf{X}_\mathbf{A} = -\sum_{k=1}^{s}\mathbf{A}^{k-1}\mathbf{Y}_{\mathbf{A}}\mathbf{c}_{k-1}\mathbf{X}_\mathbf{A}
\end{align}
as desired, because property (ii) of Proposition 2.1 in \cite{gohberg2009matrix} has $\mathbf{X}_\mathbf{A}$ uniquely defined as $\left[ 
\begin{array}{cccc} \mathbf{0} & \cdots & \mathbf{0} & \mathbf{I} \end{array}\right] \cdot \mathbf{V}^{-1}$ in our notation. This proves the theorem.
\end{proof}

\begin{theorem}[non-monic case] \label{thm:4}
	Suppose $\mathbf{a}(z) = \boldsymbol{\alpha}_{s}z^s + \boldsymbol{\alpha}_{s-1} + \cdots + \boldsymbol{\alpha}_0$ and $\boldsymbol{\alpha}_s$ might be singular. Suppose that we have a block upper Hessenberg generalized linearization $(\mathbf{A}, \mathbf{D_A})$---that is, $\mathbf{A}$ is block upper Hessenberg, $\mathbf{D_A}$ is block diagonal, each with $r \times r$ blocks, and that we have the generalized standard triple,
	\begin{equation}
		\mathbf{X_A}\left(z\mathbf{D_A} - \mathbf{A}\right)^{-1}\mathbf{D_A}\mathbf{Y_A} = \mathbf{a}^{-1}(z)\>.
	\end{equation}
	Then if $\mathbf{c}(z) = \sum_{k=0}^{s-1}\mathbf{c}_{k}z^{k}$, with each $\mathbf{c}_{i} \in \mathbb{C}^{r \times r}$, is of degree at most $s-1$, then 
\begin{equation}
	\mathbf{G} = \mathbf{A} - \sum_{k=0}^{s-1} \mathbf{A}^{k}\mathbf{Y_A}\mathbf{c}_{k}\mathbf{X_A} 
\end{equation}
is a block upper Hessenberg linearization of $\mathbf{a}(z) + \mathbf{c}(z)$, with 
\begin{equation}
	\mathbf{X_A}\left(z\mathbf{D_A} - \mathbf{G}\right)^{-1}\mathbf{D_A}\mathbf{Y_A} = \left(\mathbf{a}(z) + \mathbf{c}(z)\right)^{-1} \>.
\end{equation}
\end{theorem}

\begin{proof}
	If $\boldsymbol{\alpha}_s$ is singular, this also means that $\mathbf{D_A}$ will be singular. To find the resolvent form, we can perturb the matrix polynomial: $\mathbf{a}(z) + \varepsilon\Delta\mathbf{a}(z, \varepsilon)$, which we will define as perturbing just $\boldsymbol{\alpha}_{s}$. The generalized linearization for this new matrix polynomial is $\left(\mathbf{A}, \mathbf{D_A} + \varepsilon\mathbf{I}\right)$ (which defines $\Delta\mathbf{a}\left(z,\varepsilon\right)$ implicitly) and the standard triple is 
	\begin{equation}
		\left(\mathbf{X_A}, \left(\mathbf{D_A} + \varepsilon\mathbf{I}\right)^{-1}\mathbf{A}, \mathbf{Y_A}\right)
	\end{equation}
	which gives the resolvent form
	\begin{align}
		\left(\mathbf{a}(z) + \varepsilon\Delta\mathbf{a}(z, \varepsilon)\right)^{-1} &= \mathbf{X_A}\left(z\mathbf{I} - \left(\mathbf{D_A} + \varepsilon\mathbf{I}\right)^{-1}\mathbf{A}\right)^{-1}\mathbf{Y_A} \\
		&= \mathbf{X_A}\left(\left(\mathbf{D_A} + \varepsilon\mathbf{I}\right)^{-1}\left(z\left(\mathbf{D_A} + \varepsilon\mathbf{I}\right) - \mathbf{A}\right)\right)^{-1}\mathbf{Y_A} \\
		&= \mathbf{X_A}\left(z\left(\mathbf{D_A} + \varepsilon\mathbf{I}\right) - \mathbf{A}\right)^{-1}\left(\mathbf{D_A} + \varepsilon\mathbf{I}\right)\mathbf{Y_A} \>.
	\end{align}
	As $\varepsilon \to 0$,
	\begin{equation}
		\mathbf{a}^{-1}(z) = \mathbf{X_A}\left(z\mathbf{D_A} - \mathbf{A}\right)^{-1}\mathbf{D_A}\mathbf{Y_A} \>.
	\end{equation}
	Then using the proof from Theorem \ref{thm:3}, we find that 
	\begin{equation}
		\mathbf{G} = \mathbf{A} - \sum_{k=1}^{s} \mathbf{A}^{k-1}\mathbf{Y_A}\mathbf{c}_{k-1}\mathbf{X_A}
	\end{equation}
	is, again, the block upper Hessenberg linearization of $\mathbf{a}(z) + \mathbf{c}(z)$ with
	\begin{equation}
		\mathbf{X_A}\left(z\mathbf{D_A} - \mathbf{G}\right)^{-1}\mathbf{D_A}\mathbf{Y_A} = \left(\mathbf{a}(z) + \mathbf{c}(z)\right)^{-1} \>,
	\end{equation} 
as desired.
\end{proof}

We now come to the theorem that we wanted to prove, originally. The previous theorems are not used in the proof, although it seems that they could be. But because we want the $\mathbf{0}$ block between the $\mathbf{A}$ block and the $\mathbf{B}$ block, and because we want $\mathbf{c}_0$ in the upper right corner, it's better to apply the following direct proof.

\begin{theorem} \label{thm:5}
Let $\mathbf{a}(z)$, $\mathbf{A}$, $\mathbf{D_A}$, $\mathbf{b}(z)$, $\mathbf{B}$, $\mathbf{D_B}$ and their ancillaries be as in the previous theorems. Let $\mathbf{c}_0$, $\mathbf{d}_0 \in \mathbb{C}^{r \times r}$ be given. Then
\begin{equation}
	\mathbf{H} =
	\left[
		\begin{array}{ccc}
			\mathbf{A} & \mathbf{0} & -\mathbf{Y}_{\mathbf{A}}\mathbf{c}_0\mathbf{X}_\mathbf{B} \\
			-\mathbf{X}_\mathbf{A} & \mathbf{0} & \mathbf{0} \\
			\mathbf{0} & -\mathbf{Y}_{\mathbf{B}} & \mathbf{B}
		\end{array}
	\right]
\end{equation}
and
\begin{equation}
	\mathbf{D_H} =
	\left[
		\begin{array}{ccc}
			\mathbf{D_A} & & \\
			& \mathbf{d}_0 & \\
			& & \mathbf{D_B}
		\end{array}
	\right]
\end{equation}
linearize $\mathbf{h}(z) = z\mathbf{a}(z)\mathbf{d}_0\mathbf{b}(z) + \mathbf{c}_0$; we have
\begin{equation}
	\mathbf{X}_\mathbf{H} = 
	\left[
		\begin{array}{ccc}
			\mathbf{0} & \mathbf{0} & \mathbf{X}_\mathbf{B}
		\end{array}
	\right]
	\quad \text{and} \quad
	\mathbf{Y}_{\mathbf{H}} =
	\left[
		\begin{array}{c}
			\mathbf{Y}_\mathbf{A} \\
			\mathbf{0} \\
			\mathbf{0}
		\end{array}
	\right]
\end{equation}
making a standard triple with
\begin{equation}
	\mathbf{X}_\mathbf{H}\left(z\mathbf{D_H} - \mathbf{H}\right)^{-1}\mathbf{Y}_\mathbf{H} = \mathbf{h}^{-1}(z) \>.
\end{equation}
\end{theorem}
An explicit formula for $\left(z\mathbf{D_H} - \mathbf{H}\right)^{-1}$ will be given in the proof.
\begin{proof}
We use a compound Schur factoring, \textsl{i.e.}~use the Schur complement twice.
\begin{equation}
	z\mathbf{D_H} - \mathbf{H} = \mathbf{F}_1\mathbf{F}_2
\end{equation}
where
\begin{equation}
	\mathbf{F}_1 =
	\left[
		\begin{array}{ccc}
			z\mathbf{D_A} - \mathbf{A} & \mathbf{0} & \mathbf{0} \\
			\mathbf{X}_\mathbf{A} & \mathbf{I}_r & \mathbf{0} \\
			\mathbf{0} & \mathbf{0} & \mathbf{I}_{tr}
		\end{array}
	\right]
\end{equation}
and
\begin{equation}
	\mathbf{F}_2 =
	\left[
		\begin{array}{ccc}
			\mathbf{I}_{sr} & \mathbf{0} & \left(z\mathbf{D_A} - \mathbf{A}\right)^{-1}\mathbf{Y}_\mathbf{A}\mathbf{c}_0\mathbf{X}_\mathbf{B} \\
			\mathbf{0} & z\mathbf{d}_0 & -\mathbf{X}_\mathbf{A}\left(z\mathbf{D_A} - \mathbf{A}\right)^{-1}\mathbf{Y}_\mathbf{A}\mathbf{c}_0\mathbf{X}_\mathbf{B} \\
			\mathbf{0} & \mathbf{Y}_\mathbf{B} & z\mathbf{D_B} - \mathbf{B}
		\end{array}
	\right] \>.
\end{equation}
This is
\begin{equation}
	\mathbf{F}_2 = \mathbf{F}_3 \mathbf{F}_4
\end{equation}
with
\begin{equation}
	\mathbf{F}_3 =
	\left[
		\begin{array}{ccc}
			\mathbf{I}_{sr} & \mathbf{Q}_{j} & \left(z\mathbf{D_A} - \mathbf{A}\right)^{-1}\mathbf{Y}_\mathbf{A}\mathbf{c}_0\mathbf{X}_\mathbf{B} \\
			\mathbf{0} & \mathbf{Q}_{k} & -\mathbf{a}^{-1}\mathbf{c}_0\mathbf{X}_\mathbf{B}\left(z\mathbf{D_B} - \mathbf{B}\right)^{-1} \\
			\mathbf{0} & \mathbf{0} & \mathbf{I}_{tr}
		\end{array}
	\right] \>,
\end{equation}
where
\begin{equation}
	\mathbf{Q}_{j} = -\left(z\mathbf{D_A} - \mathbf{A}\right)^{-1}\mathbf{Y}_\mathbf{A}\mathbf{c}_0\mathbf{X}_\mathbf{B}\left(z\mathbf{D_B} - \mathbf{B}\right)^{-1}\mathbf{Y}_\mathbf{B}
\end{equation}
and
\begin{equation}
	\mathbf{Q}_{k} = z\mathbf{d}_0 + \mathbf{X}_\mathbf{A}\left(z\mathbf{D_A} - \mathbf{A}\right)^{-1}\mathbf{Y}_\mathbf{A}\mathbf{c}_0\mathbf{X}_\mathbf{B}\left(z\mathbf{D_B} - \mathbf{B}\right)^{-1}\mathbf{Y}_{\mathbf{B}} \>,
\end{equation}
and
\begin{equation}
	\mathbf{F}_4 =
	\left[
		\begin{array}{ccc}
			\mathbf{I}_{sr} & \mathbf{0} & \mathbf{0} \\
			\mathbf{0} & \mathbf{I}_r & \mathbf{0} \\
			\mathbf{0} & \mathbf{Y}_\mathbf{B} & z\mathbf{D_B} - \mathbf{B}
		\end{array}
	\right] \>.
\end{equation}
Simplifying $\mathbf{F}_3$ further,
\begin{equation}
	\mathbf{F}_3 =
	\left[
		\begin{array}{ccc}
			\mathbf{I}_{sr} & -\left(z\mathbf{D_A} - \mathbf{A}\right)^{-1}\mathbf{Y}_\mathbf{A}\mathbf{c}_0\mathbf{b}^{-1} & \mathbf{Q}_{\ell}\\
			\mathbf{0} & z\mathbf{d}_0 + \mathbf{a}^{-1}\mathbf{c}_0\mathbf{b}^{-1} & -\mathbf{a}^{-1}\mathbf{c}_0\mathbf{X}_\mathbf{B}\left(z\mathbf{D_B} - \mathbf{B}\right)^{-1} \\
			\mathbf{0} & \mathbf{0} & \mathbf{I}_{tr}
		\end{array}
	\right] \>,
\end{equation}
where,
\begin{equation}
	\mathbf{Q}_{\ell} = \left(z\mathbf{D_A} - \mathbf{A}\right)^{-1}\mathbf{Y}_\mathbf{A}\mathbf{c}_0\mathbf{X}_\mathbf{B}\left(z\mathbf{D_B} - \mathbf{B}\right)^{-1}  \>.
\end{equation}
Since
\begin{align}
	\operatorname{det}\left(z\mathbf{D_H} - \mathbf{H}\right) &= \operatorname{det}\mathbf{F}_1 \cdot \operatorname{det}\mathbf{F}_2 \\
	&= \operatorname{det} \mathbf{F}_1\cdot\operatorname{det}\mathbf{F}_3\cdot\operatorname{det}\mathbf{F}_4 \>,
\end{align}
we have
\begin{align}
	\operatorname{det}\left(z\mathbf{D_H} - \mathbf{H}\right) &= \operatorname{det}\left(z\mathbf{D_A} - \mathbf{A}\right)\operatorname{det}\left(z\mathbf{d}_0 + \mathbf{a}^{-1}\mathbf{c}_0\mathbf{b}^{-1}\right)\operatorname{det}\left(z\mathbf{D_B} - \mathbf{B}\right) \\
	&= \operatorname{det}\mathbf{a}(z)\operatorname{det}\left(z\mathbf{d}_0 + \mathbf{a}^{-1}\mathbf{c}_0\mathbf{b}^{-1}\right)\operatorname{det}\mathbf{b}(z) \\
	&=\operatorname{det}\left(z\mathbf{a}(z)\mathbf{d}_0\mathbf{b}(z) + \mathbf{c}_0\right) \\
	&=\operatorname{det}\left(\mathbf{h}(z)\right)
\end{align}
as claimed.
\end{proof}
To find the explicit form of the resolvent inverse, we invert the factors:
\begin{equation}
	\left(z\mathbf{D_H} - \mathbf{H}\right)^{-1} = \mathbf{F}_{4}^{-1}\mathbf{F}_{3}^{-1}\mathbf{F}_1^{-1} \>.
\end{equation}
\begin{align}
	\mathbf{F}_{4}^{-1} &=
	\left[
		\begin{array}{ccc}
			\mathbf{I}_{sr} & & \\
			& \mathbf{I}_r & \\
			& & \left(z\mathbf{D_B} - \mathbf{B}\right)^{-1}
		\end{array}
	\right]
	\left[
		\begin{array}{ccc}
			\mathbf{I}_{sr} & & \\
			& \mathbf{I}_r & \\
			& -\mathbf{Y}_\mathbf{B} & \mathbf{I}_{tr}
		\end{array}
	\right] \\
	&=
	\left[
		\begin{array}{ccc}
			\mathbf{I}_{sr} & & \\
			& \mathbf{I}_r & \\
			& -\left(z\mathbf{D_B} - \mathbf{B}\right)^{-1}\mathbf{Y}_\mathbf{B} & \left(z\mathbf{D_B} - \mathbf{B}\right)^{-1}
		\end{array}
	\right] \>.
\end{align}
Now (using $\boldsymbol{\alpha}$, $\boldsymbol{\beta}$, $\boldsymbol{\gamma}$ as shorthand for the relevant blocks, where $\boldsymbol{\beta}$ is regular),
\begin{align}
	\mathbf{F}_3 &=
	\left[
		\begin{array}{ccc}
			\mathbf{I} & {\boldsymbol{\alpha}} & \boldsymbol{\gamma} \\
			& \boldsymbol{\beta} & \boldsymbol{\delta} \\
			& & \mathbf{I}
		\end{array}
	\right]
	=
	\left[
		\begin{array}{ccc}
			\mathbf{I} & & \\
			& \boldsymbol{\beta} & \\
			& & \mathbf{I}
		\end{array}
	\right]
	\left[
		\begin{array}{ccc}
			\mathbf{I} & \boldsymbol{\alpha} & \boldsymbol{\gamma} \\
			& \mathbf{I} & \boldsymbol{\beta}^{-1}\boldsymbol{\delta} \\
			& & \mathbf{I}
		\end{array}
	\right] \\
	&=
	\left[
		\begin{array}{ccc}
			\mathbf{I} & & \\
			& \boldsymbol{\beta} & \\
			& & \mathbf{I}
		\end{array}
	\right]
	\left[
		\begin{array}{ccc}
			\mathbf{I} & & \boldsymbol{\gamma} \\
			& \mathbf{I} & \boldsymbol{\beta}^{-1}\boldsymbol{\delta} \\
			& & \mathbf{I}
		\end{array}
	\right]
	\left[
		\begin{array}{ccc}
			\mathbf{I} & \boldsymbol{\alpha} & \\
			& \mathbf{I} & \\
			& & \mathbf{I}
		\end{array}
	\right]
\end{align}
So
\begin{align}
	\mathbf{F}_3^{-1} &=
	\left[
		\begin{array}{ccc}
			\mathbf{I} & -\boldsymbol{\alpha} & \\
			& \mathbf{I} & \\
			& & \mathbf{I}
		\end{array}
	\right]
	\left[
		\begin{array}{ccc}
			\mathbf{I} & & -\boldsymbol{\gamma} \\
			& \mathbf{I} & -\boldsymbol{\beta}^{-1}\boldsymbol{\delta} \\
			& & \mathbf{I}
		\end{array}
	\right]
	\left[
		\begin{array}{ccc}
			\mathbf{I} & & \\
			& \boldsymbol{\beta}^{-1} & \\
			& & \mathbf{I}
		\end{array}
	\right] \\
	&=
	\left[
		\begin{array}{ccc}
			\mathbf{I} & -\boldsymbol{\alpha} & -\boldsymbol{\gamma} + \boldsymbol{\alpha}\boldsymbol{\beta}^{-1}\boldsymbol{\delta} \\
			& \mathbf{I} & -\boldsymbol{\beta}^{-1}\boldsymbol{\delta} \\
			& & \mathbf{I}
		\end{array}
	\right]
	\left[
		\begin{array}{ccc}
			\mathbf{I} & & \\
			& \boldsymbol{\beta}^{-1} & \\
			& & \mathbf{I}
		\end{array}
	\right] \\
	&=
	\left[
		\begin{array}{ccc}
			\mathbf{I} & -\boldsymbol{\alpha\beta}^{-1} & -\boldsymbol{\gamma} + \boldsymbol{\alpha\beta}^{-1}\boldsymbol{\delta} \\
			& \boldsymbol{\beta}^{-1} & -\boldsymbol{\beta}^{-1}\boldsymbol{\delta} \\
			& & \mathbf{I}
		\end{array}
	\right]
\end{align}
So
\begin{equation}
	\mathbf{F}_3^{-1} =
	\left[
		\begin{array}{ccc}
			\mathbf{I}_{sr} & \left(z\mathbf{D_A} - \mathbf{A}\right)^{-1}\mathbf{Y}_\mathbf{A}\mathbf{c}_0\mathbf{h}^{-1}\mathbf{a} & \mathbf{U} \\
			& \mathbf{bh}^{-1}\mathbf{a} & \mathbf{bh}^{-1}\mathbf{c}_0\mathbf{X}_\mathbf{B}\left(z\mathbf{D_B} - \mathbf{B}\right)^{-1} \\
			& & \mathbf{I}_{tr}
		\end{array}
	\right]
\end{equation}
using $z\mathbf{d}_0 + \mathbf{a}^{-1}\mathbf{c}_0\mathbf{b}^{-1} = \mathbf{a}^{-1}\left(z\mathbf{a}\mathbf{d}_0\mathbf{b} + \mathbf{c}_0\right)\mathbf{b}^{-1} = \mathbf{a}^{-1}\mathbf{h}\mathbf{b}^{-1}$ so
\begin{equation}
	\left(z\mathbf{d}_0 + \mathbf{a}^{-1}\mathbf{c}_0\mathbf{b}^{-1}\right)^{-1} = \mathbf{bh}^{-1}\mathbf{a}
\end{equation}
and
\begin{align}
	-\left(z\mathbf{D_A} - \right.&\left.\mathbf{A}\right)^{-1}\mathbf{Y}_\mathbf{A}\mathbf{c}_0\mathbf{X}_\mathbf{B}\left(z\mathbf{D_B} - \mathbf{B}\right)^{-1} \\ &+ \left(z\mathbf{D_A} - \mathbf{A}\right)^{-1}\mathbf{Y}_\mathbf{A}\mathbf{c}_0\mathbf{h}^{-1}\mathbf{C}_0\mathbf{X}_\mathbf{B}\left(z\mathbf{D_B} - \mathbf{B}\right)^{-1} = \mathbf{U} \\
	&= \left(z\mathbf{D_A} - \mathbf{A}\right)^{-1}\mathbf{Y}_\mathbf{A}\left[\mathbf{c}_0\mathbf{h}^{-1}\mathbf{c}_0 - \mathbf{c}_0\right]\mathbf{X}_\mathbf{B}\left(z\mathbf{D_B} - \mathbf{B}\right)^{-1} \>.
\end{align}
In the next section we will use this at $z = 0$ if $\mathbf{c}_0$ is invertible to show $\mathbf{U} = \mathbf{0}$ ($sr$ by $tr$ block). Also,
\begin{equation}
	\mathbf{F}_1^{-1} =
	\left[
		\begin{array}{ccc}
			\left(z\mathbf{D_A} - \mathbf{A}\right)^{-1} & & \\
			-\mathbf{X}_\mathbf{A}\left(z\mathbf{D_A} - \mathbf{A}\right)^{-1} & \mathbf{I}_r & \\
			& & \mathbf{I}_{tr}
		\end{array}
	\right] \>.
\end{equation}
Therefore $\left(z\mathbf{I} - \mathbf{H}\right)^{-1}$ is $\mathbf{F}_4^{-1}\mathbf{F}_3^{-1}\mathbf{F}_{1}^{-1}$. Now
\begin{align}
	&\mathbf{F}_4^{-1}\mathbf{F}_3^{-1} = \\
	&\left[
		\begin{array}{ccc}
			\mathbf{I}_{sr} & \left(z\mathbf{D_A} - \mathbf{A}\right)^{-1}\mathbf{Y}_\mathbf{A}\mathbf{c}_0\mathbf{h}^{-1}\mathbf{a} & \mathbf{U} \\
			\mathbf{0} & \mathbf{b}\mathbf{h}^{-1}\mathbf{a} & \mathbf{b}\mathbf{h}^{-1}\mathbf{c}_0\mathbf{X}_\mathbf{B}\left(z\mathbf{D_B} - \mathbf{B}\right)^{-1} \\
			\mathbf{0} & -\left(z\mathbf{D_B} - \mathbf{B}\right)^{-1}\mathbf{Y}_\mathbf{B}\mathbf{b}\mathbf{h}^{-1}\mathbf{a} & \mathbf{R}_{33}
		\end{array}
	\right]
\end{align}
where 
\begin{equation}
	\mathbf{R}_{33} = \left(z\mathbf{D_B} - \mathbf{B}\right)^{-1} - \left(z\mathbf{D_B} - \mathbf{B}\right)^{-1}\mathbf{Y}_\mathbf{B}\mathbf{b}\mathbf{h}^{-1}\mathbf{c}_0\mathbf{X}_\mathbf{B}\left(z\mathbf{D_B} - \mathbf{B}\right)^{-1} \>. 
\end{equation}
Therefore $\mathbf{F}_4^{-1}\mathbf{F}_3^{-1}\mathbf{F}_1^{-1}$ is
\begin{equation}
	\left[
		\begin{array}{ccc}
			\mathbf{R}_{11} & \left(z\mathbf{D_A} - \mathbf{A}\right)^{-1}\mathbf{Y}_\mathbf{A}\mathbf{c}_0\mathbf{b}^{-1}\mathbf{a} & \mathbf{U} \\
			\mathbf{R}_{21} & \mathbf{b}\mathbf{h}^{-1}\mathbf{a} & \mathbf{b}\mathbf{h}^{-1}\mathbf{c}_0\mathbf{X}_{\mathbf{B}}\left(z\mathbf{D_B} - \mathbf{B}\right)^{-1} \\
			 \mathbf{R}_{31}& -\left(z\mathbf{D_B} - \mathbf{B}\right)^{-1}\mathbf{Y}_\mathbf{B}\mathbf{b}\mathbf{h}^{-1}\mathbf{a} & \mathbf{R}_{33}
		\end{array}
	\right]
\end{equation}
where 
\begin{align}
	\mathbf{R}_{11} &= \left(z\mathbf{D_A} - \mathbf{A}\right)^{-1} - \left(z\mathbf{D_A} - \mathbf{A}\right)^{-1}\mathbf{Y}_\mathbf{A}\mathbf{c}_0\mathbf{h}^{-1}\mathbf{a}\mathbf{X}_\mathbf{A}\left(z\mathbf{D_A} - \mathbf{A}\right) \>, \\
	\mathbf{R}_{21} &= -\mathbf{b}\mathbf{h}^{-1}\mathbf{a}\mathbf{X}_\mathbf{A}\left(z\mathbf{D_A} - \mathbf{A}\right)^{-1} \>, \\
	 \mathbf{R}_{31} &= \left(z\mathbf{D_B} - \mathbf{B}\right)^{-1}\mathbf{Y}_\mathbf{B}\mathbf{b}\mathbf{h}^{-1}\mathbf{a}\mathbf{X}_\mathbf{A}\left(z\mathbf{D_A} - \mathbf{A}\right)^{-1} \>.
\end{align}
Moreover,
\begin{align}
	\left[
		\begin{array}{ccc}
			\mathbf{0} & \mathbf{0} & \mathbf{X}_{B}
		\end{array}
	\right]
	\left(z\mathbf{D_H} - \mathbf{H} \right)^{-1}
	\left[
		\begin{array}{c}
			\mathbf{Y}_\mathbf{A} \\
			\mathbf{0} \\
			\mathbf{0}
		\end{array}
	\right] &= \mathbf{b}^{-1}(z)\mathbf{b}(z)\mathbf{h}^{-1}(z)\mathbf{a}(z)\mathbf{a}^{-1}(z) \\
	&= \mathbf{h}^{-1}(z) \>,
\end{align}
as desired.

\section{Implications}
Consider first the Mandelbrot matrices $\mathbf{M}_n$ from Section \ref{sec:3}. Here $r$ is just $1$, and we may deduce a sequence of facts, as follows.
\begin{lemma}\label{lmm:1}
The dimension of $\mathbf{M}_n$ is $d_n \times d_n$, where $d_n = 2^{n-1} - 1$.
\end{lemma}
\begin{proof}
Simple induction beginning with $d_2 = 1$ and $d_{n+1} = 2d_n + 1$.
\end{proof}
\begin{lemma}\label{lmm:2}
$\mathbf{X}_n$ and $\mathbf{Y}_n$ are simply $\mathbf{e}_{d_n}^{T} = \left[\begin{array}{cccc}0 &\cdots & 0 & 1 \end{array}\right]$ and $\mathbf{e}_1$ where $\mathbf{e}_1^{T} = \left[\begin{array}{cccc}1 & 0 & \cdots & 0 \end{array}\right]$.
\end{lemma}
\begin{proof}
Again induction, beginning with $\mathbf{M}_2$:
\begin{equation}
	p_2 = z + 1 \quad \Rightarrow \quad \left[ 1 \right]\left(z + 1\right)^{-1}\left[1\right] = p_2^{-1}
\end{equation}
and 
\begin{equation}
	x_{n+1} = \left[\begin{array}{ccc}\text{zeros(size($x_n$))} & 0 & x_n\end{array}\right]
\end{equation}
while 
\begin{equation}
	y_{n+1} = \left[\begin{array}{c}y_n \\0 \\ \text{zeros(size($y_n$))}\end{array}\right]
\end{equation}	
by Theorem \ref{thm:4}.
\end{proof}
\begin{lemma}\label{lmm:3}
The bottom left corner of $\mathbf{M}_{n}^{-1}$ is always $-1$.
\end{lemma}
\begin{proof}
We have several proofs for this fact, most simply using the minor of the top right corner; but we will shortly want $\mathbf{M}_n^{-1}$ explicitly and so we compute it here. We note that $\mathbf{M}_2$ is invertible, $\mathbf{M}_2^{-1} = \left[ -1\right]$, and that $\mathbf{c}_0 = 1$ is always invertible. Thus by induction $\mathbf{M}_{n+1}$ is invertible because $\mathbf{M}_n$ is. We have, by specializing the resolvent inverse from Theorem \ref{thm:4}, that the bottom left block of $\left(\mathbf{0}\cdot\mathbf{I} - \mathbf{M}_{n+1}\right)^{-1}$ is
\begin{align}
	\mathbf{M}_n^{-1}\mathbf{Y}_{n}\cdot 1 \cdot 1^{-1} \cdot 1 \cdot \mathbf{X}_n \cdot \mathbf{M}_{n}^{-1} &= \mathbf{M}_n^{-1}\left[\begin{array}{c}1 \\ 0 \\ \vdots \end{array}\right]\left[\begin{array}{ccc}0 & \cdots & 1 \end{array} \right] \mathbf{M}_{n}^{-1} \\
	&= \mathcal{C}_n\mathcal{R}_n
\end{align}
where $\mathcal{C}_n$ is the first column of $\mathbf{M}_n^{-1}$ and $\mathcal{R}_n$ is the last row of $\mathbf{M}_n^{-1}$.

By the inductive hypothesis, the bottom left corner of this block is
\begin{equation}
	\left[
		\begin{array}{ccccc}
		& & & &\\
		\\
		\\
		+1
		\end{array}
	\right]
\end{equation} 
because $(-1)(-1) = +1$. Remember this is the bottom left block of $\left(-\mathbf{M}_{n+1}\right)^{-1}$; thus if the bottom left corners of $\mathbf{M}_n^{-1}$ are $-1$, so is the bottom left corner of $\left(\mathbf{M}_{n+1}\right)^{-1}$.
\end{proof}
\begin{lemma} \label{lmm:4}
The upper left block of $\mathbf{M}_{n+1}^{-1}$ is the same as the lower right block; both are
\begin{equation}
	\mathbf{M}_n^{-1} + \mathbf{M}_n^{-1}\mathbf{Y}_n\cdot\mathbf{X}_n\mathbf{M}_n^{-1} = \mathbf{M}_n^{-1} + \mathcal{C}_n\mathcal{R}_n \>.
\end{equation}
\end{lemma}
The proof is simple computation.
\begin{lemma}\label{lmm:5}
The first column and the last row of the blocks in Lemma \ref{lmm:4} are zero.
\end{lemma}
\begin{proof}
The left column of $\mathcal{C}_n\mathcal{R}_n$ is $-\mathcal{C}_n$ because the left element of $\mathcal{R}_n$ is $-1$ by Lemma \ref{lmm:3}. Thus the left column of $\mathbf{M}_n^{-1}\mathcal{C}_n\mathcal{R}_n$ is zero. Similarly the last row of $\mathcal{C}_n\mathcal{R}_n$ is $-\mathcal{R}_n$, leading to the same conclusion.
\end{proof}
\begin{lemma}\label{lmm:6}
For $n \geq 3$, the lower left block of $\mathbf{M}_n^{-1} + \mathcal{C}_n\mathcal{R}_n$ is a $1 + d_{n-1} \times 1 + d_{n-1}$ block of zeros, and all other blocks of $\mathbf{M}_{n+1}^{-1}$ are untouched.
\end{lemma}
\begin{proof}
Consider first $\mathbf{M}_3 = \left[\begin{array}{ccc}-1 & \phantom{-}0 & -1 \\ -1 & \phantom{-}0 & \phantom{-}0 \\ \phantom{-}0 & -1 & -1 \end{array}\right]$ and $\mathbf{M}_3^{-1} = \left[\begin{array}{ccc}\phantom{-}0 & -1 & \phantom{-}0 \\ \phantom{-}1 & -1 & -1 \\ -1 & \phantom{-}1 & \phantom{-}0 \end{array}\right]$. Then $\mathcal{C}_3 = \left[\begin{array}{c}\phantom{-}0 \\ \phantom{-}1 \\ -1 \end{array}\right]$ and $\mathcal{R}_3 = \left[ \begin{array}{ccc}-1 & \phantom{-}1 & \phantom{-}0 \end{array}\right]$ so 
\begin{equation}
	\mathcal{C}_3\mathcal{R}_3 =
	\left[
		\begin{array}{ccc}
			\phantom{-}0 & \phantom{-}0 & 0 \\
			-1 & \phantom{-}1 & 0 \\
			\phantom{-}1 & -1 & 0
		\end{array}
	\right] \>,
	\quad
	\mathbf{M}_3^{-1} + \mathcal{C}_3\mathcal{R}_3 =
	\left[
		\begin{array}{ccc}
			0 & -1 & \phantom{-}0 \\
			0 & \phantom{-}0 & -1 \\
			0 & \phantom{-}0 & \phantom{-}0
		\end{array}
	\right] \>.
\end{equation}
Note that the first $d_n$ entries of $\mathcal{C}_n$ are $0$ and that the last $d_n$ entries of $\mathcal{R}_n$ are $0$. Note that as in the proof of Lemma \ref{lmm:3}, the bottom left $d_n \times d_n$ block of $\mathbf{M}_{n+1}^{-1}$ is just $\mathcal{C}_n\mathcal{R}_n$, and by specializing the resolvent formula the row just above that is $-\mathcal{R}_n$; similarly the column beside that block is $\mathcal{R}_n$; similarly the column beside that block is $\mathcal{C}_n$. Indeed
\begin{equation}
	\mathbf{M}_{n+1}^{-1} =
	\left[
		\begin{array}{ccc}
			\mathbf{M}_{n+1}^{-1} + \mathcal{C}_n\mathcal{R}_n & \mathcal{C}_n & \mathbf{0} \\
			-\mathcal{R}_n & -1 & \mathcal{R}_n \\
			-\mathcal{C}_n\mathcal{R}_n & -\mathcal{C}_n & \mathbf{M}_n^{-1} + \mathcal{C}_n\mathcal{R}_n
		\end{array}
	\right] \>.
\end{equation}
We have established in Lemma \ref{lmm:5} that the bottom right block has a zero last row and that the upper left block has a zero first column. Thus
\begin{equation}
	\mathcal{C}_{n+1} = 
	\left[
		\begin{array}{c}
			0 \\
			1 \\
			\mathcal{C}_n
		\end{array}
	\right]
	\quad \text{and} \quad
	\mathcal{R}_{n+1} =
	\left[
		\begin{array}{ccc}
			\mathcal{R}_n & 1 & 0
		\end{array}
	\right] \>.
\end{equation}
Therefore 
\begin{equation}
	\mathcal{C}_{n+1}\mathcal{R}_{n+1} = \left[\begin{array}{ccc} 0 & 0 & 0 \\ \mathcal{R}_n & 1 & 0 \\ 
	\mathcal{C}_n\mathcal{R}_n & \mathcal{C}_n & 0\end{array}\right]
\end{equation}
and 
\begin{equation}
	\mathbf{M}_{n+1}^{-1} + \mathcal{C}_{n+1}\mathcal{R}_{n+1} = \left[\begin{array}{ccc}\mathbf{M}_n^{-1}\mathcal{C}_n\mathcal{R}_n & \mathcal{C}_n & 0 \\ 0 & 0 & \mathcal{R}_n \\ 0 & 0 & \mathbf{M}_n^{-1} + \mathcal{C}_n\mathcal{R}_n\end{array}\right]
\end{equation}
establishing the claim by induction.
\end{proof}
\begin{definition}
A matrix family is \textbf{Bohemian} if its entries come from a single discrete (and hence bounded) set. Here the set is just $\{-1, 0, 1\}$. The name comes from ``\textbf{Bo}unded \textbf{He}ight \textbf{M}atrix of Integers."
\end{definition}
\begin{lemma}\label{lmm:7}
The Mandelbrot matrices are Bohemian, with height\footnote{$\mathrm{height}(\mathbf{A}) := ||\mathrm{vec}(\mathbf{A}) ||_{\infty}$ is the largest entry of $\left|\mathbf{A}\right|$, where $\left|\mathbf{A}\right|$ means the matrix whose entries are the absolute values of the entries of $\mathbf{A}$.} 1. Indeed the only entries are $0$ or $-1$.
\end{lemma}
\begin{proof}
Induction.
\end{proof}
\begin{definition}
A matrix family has \textbf{rhapsody} if it is Bohemian and its inverse is also Bohemian with the same height. 
\end{definition}
\begin{theorem}
The Mandelbrot matrices have rhapsody.
\end{theorem}
\begin{proof}
By induction using the previous lemmas. Clearly the entries of $\mathcal{C}_n$ and $\mathcal{R}_n$ are $+1$, $-1$, or $0$; thus $\mathcal{C}_n\mathcal{R}_n$ has height 1. Since the contribution of $\mathcal{C}_n\mathcal{R}_n$ to $\mathbf{M}_n^{-1}$ in $\mathbf{M}_n^{-1} + \mathcal{C}_n\mathcal{R}_n$ was entirely removing the lower left $1+d_{n-1}$ by $1 + d_{n-1}$ block, and did not touch the other entries, each block remains of height 1. 
\end{proof}

\section{First Matrix Polynomial Experiments}
\label{sec:experiments}
To test these ideas we examine a family of matrix polynomials that we have artificially created for the purpose. We use the following recursive construction. Put
\begin{equation}
	\mathbf{h}_{1} = z\mathbf{I} + \mathbf{c}_{0}
\end{equation}
and for $k \geq 0$
\begin{equation}
	\mathbf{h}_{k+1}(z) = z\mathbf{h}_{k}^{2}(z) + \mathbf{c}_{k}(z)
	\label{eq:exp1}
\end{equation}
where $\mathbf{c}_{k}(z) \in \mathbb{C}^{4\times 4}$ are nonsingular upper Hessenberg matrices with zero diagonal and entries $-1$ on the subdiagonal. We choose these matrices $\mathbf{c}_k$ in advance, not all the same. This gives a ``Mandelbrot-like'' flavour to the construction. Notice that for $k \geq 1$ $\mathrm{deg}\ \mathbf{h}_{k}(z) = 2^{k} - 1$, and its dimension is $4 \times 4$ for every $k$. The linearization of Theorem \ref{thm:4} gives matrices $\mathbf{H}_{k}$ of dimension $4\cdot\left(2^{k}- 1\right)$ by $4\cdot\left(2^{k}- 1\right)$. Our experiments covered various choices of the $\mathbf{c}_{k}$ and dimensions up to $16380\times 16380$. 

The matrices $\mathbf{c}_{k}$ that we used are
\begin{gather}
	\left[ \begin {array}{cccc} 0&-1&-1&-1\\ \noalign{\medskip}-1&0&0&1
\\ \noalign{\medskip}0&-1&0&1\\ \noalign{\medskip}0&0&-1&0\end {array}
 \right] , \left[ \begin {array}{cccc} 0&-1&-1&-1\\ \noalign{\medskip}
-1&0&1&1\\ \noalign{\medskip}0&-1&0&0\\ \noalign{\medskip}0&0&-1&0
\end {array} \right] , \left[ \begin {array}{cccc} 0&-1&-1&-1
\\ \noalign{\medskip}-1&0&0&0\\ \noalign{\medskip}0&-1&0&1
\\ \noalign{\medskip}0&0&-1&0\end {array} \right] , \\
\left[ 
\begin {array}{cccc} 0&-1&-1&-1\\ \noalign{\medskip}-1&0&1&0
\\ \noalign{\medskip}0&-1&0&0\\ \noalign{\medskip}0&0&-1&0\end {array}
 \right] , \left[ \begin {array}{cccc} 0&-1&-1&-1\\ \noalign{\medskip}
-1&0&0&-1\\ \noalign{\medskip}0&-1&0&1\\ \noalign{\medskip}0&0&-1&0
\end {array} \right] , \left[ \begin {array}{cccc} 0&-1&-1&-1
\\ \noalign{\medskip}-1&0&1&-1\\ \noalign{\medskip}0&-1&0&0
\\ \noalign{\medskip}0&0&-1&0\end {array} \right] , \\
 \left[\begin {array}{cccc} 0&-1&-1&0\\ \noalign{\medskip}-1&0&-1&1
\\ \noalign{\medskip}0&-1&0&-1\\ \noalign{\medskip}0&0&-1&0
\end {array} \right] , \left[ \begin {array}{cccc} 0&-1&-1&-1
\\ \noalign{\medskip}-1&0&-1&1\\ \noalign{\medskip}0&-1&0&0
\\ \noalign{\medskip}0&0&-1&0\end {array} \right] , \left[ \begin {array}{cccc} 0&-1&-1&0\\ \noalign{\medskip}-1&0&0&1 \\ \noalign{\medskip}0&-1&0&-1\\ \noalign{\medskip}0&0&-1&0
\end {array} \right] , \\ \left[ \begin {array}{cccc} 0&-1&-1&-1
\\ \noalign{\medskip}-1&0&-1&1\\ \noalign{\medskip}0&-1&0&1
\\ \noalign{\medskip}0&0&-1&0\end {array} \right] , \left[ 
\begin {array}{cccc} 0&-1&-1&-1\\ \noalign{\medskip}-1&0&0&1
\\ \noalign{\medskip}0&-1&0&0\\ \noalign{\medskip}0&0&-1&0\end {array}
 \right] , \left[ \begin {array}{cccc} 0&-1&-1&0\\ \noalign{\medskip}-
1&0&1&1\\ \noalign{\medskip}0&-1&0&-1\\ \noalign{\medskip}0&0&-1&0
\end {array} \right] \>.
\end{gather}
Larger experiments are of course possible and desirable.

\begin{figure}[h]
	\centering
	\includegraphics[width=0.7\textwidth]{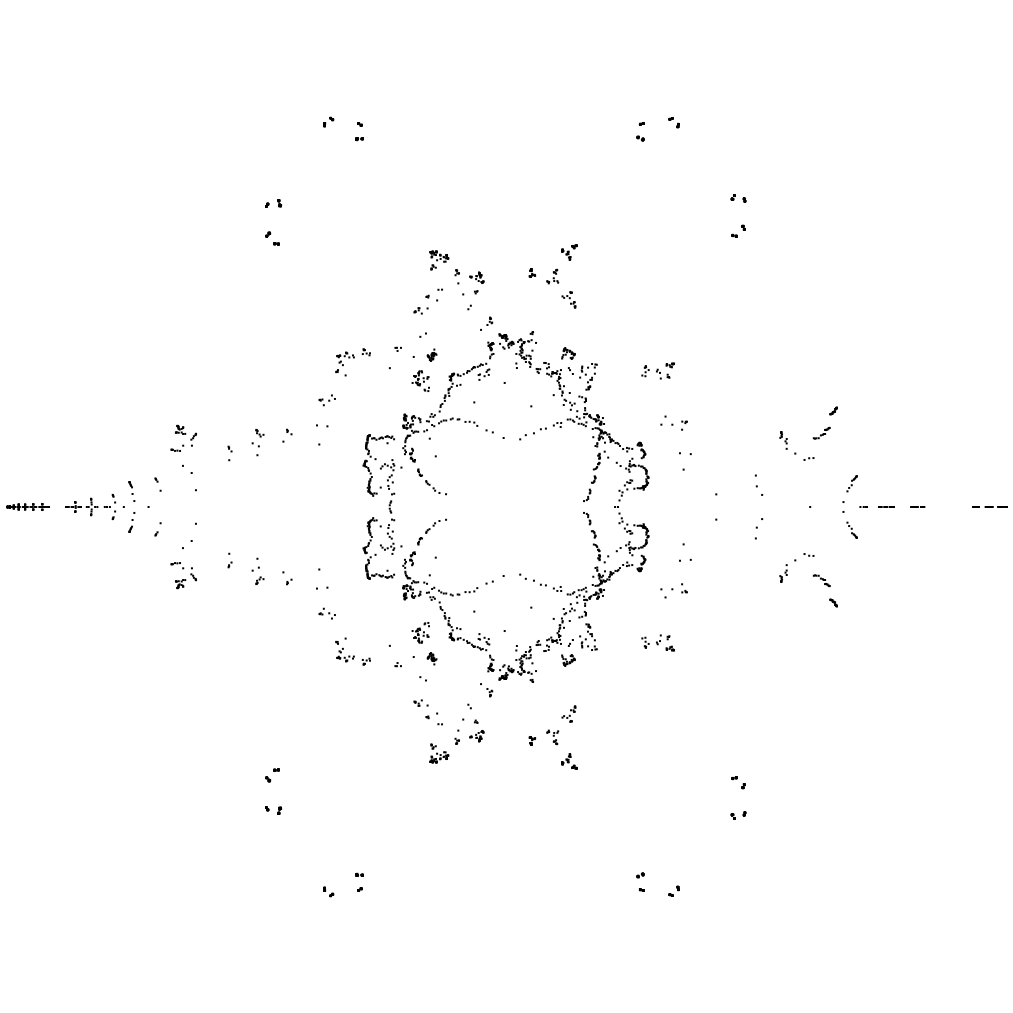}
	\caption{Eigenvalues of a $4092 \times 4092$ matrix. For details, refer to equation \eqref{eq:exp1}}
	\label{fig:example}
\end{figure}

We exhibit the eigenvalues of one $4092\times 4092$ ($k = 10$) matrix in figure \ref{fig:example}. We compared the computed eigenvalues (computed using Maple's \\ \texttt{LinearAlgebra:-Eigenvalues} routine, which calls an implementation of LAPACK via the NAG library) with the roots of the characteristic polynomials $p_{k}(z) = \mathrm{det}(\mathbf{h}_{k}(z))$ computed by Maple's built-in solver \texttt{fsolve} (refer to \cite{char2012first}) which is slow but quite reliable. Because the height of the exactly-computed characteristic polynomial reached $10^{234}$, solving the polynomial using \texttt{fsolve} required multiple precision, which is slow. To compute the residual, we computed the singular values of our matrix polynomial, $\mathbf{h}_{10}(\xi_{i})$ for each of the eigenvalues, $\xi_{i}$, and divided the smallest singular value by the largest singular value for each case. The residual $\sigma_{4}/\sigma_{1}$ of $\mathbf{h}_{4}(\lambda)$ in any polynomial eigenvalue $\lambda$ was never more than $\sim4 \times 10^{-13}$. Table \ref{tab:time} shows the time taken to compute the eigenvalues and time take to compute the roots using Maple's \texttt{fsolve} (using a machine with 32 GB of memory). Eigenvalue computation of the linearization was always the fastest taking only 94.782 seconds for the $k = 10$ case. Unfortunately, we had to kill the job on Maple after a week for the $k = 10$ case.

\begin{table}
	\centering
	\begin{tabular}{|c|c|c|c|}
		\hline
		$k$ & Dimension & \texttt{Eigenvalues} (s) & \texttt{fsolve} (s) \\ 
		\hline
		5 & 124 & 0.047 & 0.531 \\
		6 & 252 & 0.125 & 2.313 \\
		7 & 508 & 0.640 & 176.203 \\
		8 & 1020 & 2.375 &  829.093\\
		9 & 2044 & 13.469 & 80242.078 \\
		10 & 4092 & 94.782 & $-$ \\
		11 & 8188 & 715.109 & $-$ \\
		12 & 16380 & 6367.703 & $-$ \\
		\hline
	\end{tabular}
	\caption{Times and residuals of eigenvalue computation of the algebraic linearizations using Maple. The polynomial solver \texttt{fsolve} takes so long because the heights of the characteristic polynomials grow exponentially in the dimension. The eigenvalue solver has no difficulty, because the matrix height is constant.}
	\label{tab:time}
\end{table}

In another experiment, for a specialized example, we compared the accuracy of the eigenvalues from our companion construction and the eigenvalues from the Frobenius companion construction in \textsc{Matlab}. In comparison to the previous experiment, we used a lower degree matrix polynomial
\begin{equation}
	\mathbf{H}(z) = z\mathbf{a}(z)\mathbf{b}(z) + \mathbf{c}_{0} \>,
\end{equation}
where
\begin{align}
	\mathbf{a}(z) = \sum_{k = 0}^{3} z^{k}\mathbf{A}_{k} \>, \\
	\mathbf{b}(z) = \sum_{k=0}^{3} z^{k}\mathbf{B}_{k} \>,
\end{align}
and $\mathbf{c}_{0} = \mathbf{I}_{5}$. The matrices $\mathbf{A}_{k}$ that we used here were chosen by calling Maple's \texttt{RandomMatrix} function. For reference, the ones we used were
\begin{align}
	\mathbf{A}_{0} &=
	\begin{bmatrix}
		-81&-98&-76&-4&29\\ 
		-38&-77&-72&27&44\\ 
		-18&57&-2&8&92 \\ 
		87&27&-32&69&-31\\ 
		33&-93&-74& 99&67
	\end{bmatrix} \>, \\
	\mathbf{A}_{1} &=
	\begin{bmatrix}
		76&20&31&94&-16\\ 
		-44&-61&-50&12&-9\\ 
		24&-48&-80&-2&-50\\ 
		65&77&43&50&-22\\ 
		86&9&25&10&45
	\end{bmatrix} \>, \\
	\mathbf{A}_{2} &=
	\begin{bmatrix}
		70&82&12&22&60\\ 
		-32&72&-62&14&-95\\ 
		-1&42&-33&16&-20\\ 
		52&18&-68&9&-25\\ 
		-13&-59&-67&99&51
	\end{bmatrix} \>, \\
	\mathbf{A}_{3} &=
	\begin{bmatrix}
		-38&-63&12&21&-82\\ 
		91&-26&45&90&-70\\ 
		-1&30&-14&80&41\\ 
		63&10&60&19&91\\ 
		-23&22&-35&88&29
	\end{bmatrix} \>.
\end{align}
We then randomly assigned 
\begin{equation}
	\mathbf{B}_{0} =
	\begin{bmatrix}
		-15&10&-83&10&-4\\ 
		2&-44&9&-61&5\\ 
		-88&26&88&-26&-91\\ 
		99&-3&95&-20&-44\\ 
		-59&-62&63&-78&-38
	\end{bmatrix} \>,
\end{equation}
and chose the rest of the $\mathbf{B}_{k}$ to be 
\begin{align}
	\mathbf{B}_{3} &= \mathbf{A}_{3}^{-1} \\
	\mathbf{B}_{2} &= -\mathbf{A}_{3}^{-1}\mathbf{A}_{2}\mathbf{B}_{3} \\
	\mathbf{B}_{1} &= -\mathbf{A}_{3}^{-1}\left(\mathbf{A}_{1}\mathbf{B}_{3} + \mathbf{A}_{2}\mathbf{B}_{2}\right)
\end{align}
so that some of the coefficients of $\mathbf{H}(z)$, when expressed in the monomial basis, would be $\mathbf{0}$. However, since we are computing these coefficients numerically, rounding errors would be introduced, resulting in loss of accuracy as we will see in the residuals.

In order to construct the algebraic linearization of $\mathbf{H}(z)$, we need the linearizations of both $\mathbf{a}(z)$ and $\mathbf{b}(z)$. We decided to use the Frobenius companion construction for these smaller companions, since the coefficients were readily available to use. The rest then follows the construction described in this paper. This suggests the idea that we can potentially mix different polynomial bases using our construction, which will be elaborated on in the next example.


We computed the residuals (as described in our previous example) to compare the accuracy of the two results. We found that the largest residual for the eigenvalues of the algebraic linearization is approximately $7.8 \times 10^{-12}$ and the largest residual for the Frobenius companion matrix is approximately $7.0 \times 10^{-9}$, around 900 times larger. This suggests that the algebraic linearization may be more numerically stable.

For our third example, we show that one can mix different polynomial bases together. All that is needed is a standard triple for $\mathbf{a}(z)$ and another for $\mathbf{b}(z)$, like so:
\begin{align}
	\mathbf{X_A}(z\mathbf{A}_{1} - \mathbf{A}_{0})^{-1}\mathbf{Y_A} &= \mathbf{a}^{-1}(z) \\
	\mathbf{X_B}(z\mathbf{B}_{1} - \mathbf{B}_{0})^{-1}\mathbf{Y_B} &= \mathbf{b}^{-1}(z) \>.
\end{align}
For instance, suppose $\mathbf{a}(z)$ is expressed in the barycentric Lagrange basis, as follows:
\begin{equation}
	\mathbf{a}(z) = w(z) \sum_{k=0}^{n} \dfrac{\beta_{k}\mathbf{a}_{k}}{z - \tau_{k}} \quad \mathbf{a}_{k} \in \mathbb{C}^{r \times r}
\end{equation}
where the $\tau_{k}$ are distinct nodes, the node polynomial is $w(z) = \prod_{k = 0}^{n}(z - \tau_{k})$, and the barycentric weights $\beta_{k}$ come from the partial fraction decomposition
\begin{equation}
	\dfrac{1}{w(z)} = \sum_{k=0}^{n}\dfrac{\beta_{k}}{z - \tau_{k}} \>.
\end{equation}

Then there are several choices for linearizations of $\mathbf{a}(z)$ without needing to change bases. See \cite{amiraslani2008linearization} or \cite{van2015linearization}. In 2004, RMC implemented the following linearization in Maple\cite{corless2004generalized}: if
\begin{equation}
	\mathbf{A}_{0} = 
	\begin{bmatrix}
		-\tau_{0}\mathbf{I} & & & & \mathbf{a}_{0}^\mathrm{T} \\
		& -\tau_{1}\mathbf{I} & & & \mathbf{a}_{1}^\mathrm{T} \\
		& & \ddots & & \vdots \\
		& & & -\tau_{n}\mathbf{I} & \mathbf{a}_{n}^\mathrm{T} \\
		-\beta_{0}\mathbf{I} & -\beta_{1}\mathbf{I} & \cdots & -\beta_{n}\mathbf{I} & \mathbf{0}
	\end{bmatrix}
	\quad
	\mathbf{A}_{1} =
	\begin{bmatrix}
		-\mathbf{I} & & & & \\
		& -\mathbf{I} & & & \\
		& & \ddots & & \\
		& & & -\mathbf{I} & \\
		& & & & \mathbf{0} 
	\end{bmatrix}
\end{equation}
then $\mathrm{det}(\mathbf{A}_{0} - z\mathbf{A}_{1}) = \mathrm{det}\left(\mathbf{a}(z)^{\mathrm{T}}\right) = \mathrm{det}(\mathbf{a}(z))$. Putting the zero blocks in the lower left corner is not as numerically stable as using linearizations with the zero blocks in the upper left corner (see \cite{lawrence2014stability}) but we'll use the existing software. The transpose also complicates this example, but not much.

It can be shown that
\begin{equation}
	\mathbf{X_A} = 
	\begin{bmatrix}
		\mathbf{0} & \mathbf{0} & \cdots & \mathbf{0} & \mathbf{I}
	\end{bmatrix}
\end{equation}
and
\begin{equation}
	\mathbf{Y_A} =
	\begin{bmatrix}
		\mathbf{I} & \mathbf{I} & \cdots & \mathbf{I} & \mathbf{0}
	\end{bmatrix}^{\mathbf{T}}
\end{equation}
give $\mathbf{X_A}(z\mathbf{A}_{1} - \mathbf{A}_{0})^{-1} \mathbf{Y_A} = \mathbf{a}^{-1}(z)$ (note the sign reversal).

For $\mathbf{b}(z)$, we choose the Chebyshev basis. One could equally well choose the Legendre basis (implemented in Maple as \texttt{JacobiP(k, 0, 0, x)}) or any other bases. The generalized companion matrix (``colleague'' matrices of \cite{good1961colleague} and of \cite{specht1956lage} independently) give the linearization of $\mathbf{b}_{0}T_{0}(x) + \mathbf{b}_{1}T_{1}(x) + \mathbf{b}_{2}T_{2}(x) + \mathbf{b}_{3}T_{3}(x) + \mathbf{b}_{4}T_{4}(x) + \mathbf{b}_{5}T_{5}(x)$ as
\begin{equation}
	\mathbf{B}_{0} =
	\begin{bmatrix}
		\mathbf{0} & \frac{1}{2}\mathbf{I} & \mathbf{0} & \mathbf{0} & -\mathbf{b}_0 \\
		\mathbf{I} & \mathbf{0} & \frac{1}{2}\mathbf{I} & \mathbf{0} & -\mathbf{b}_1 \\
		& \frac{1}{2}\mathbf{I} & \mathbf{0} & \frac{1}{2}\mathbf{I} & -\mathbf{b}_{2} \\
		& & \frac{1}{2}\mathbf{I} & \mathbf{0} & -\mathbf{b}_{3} + \mathbf{b}_{5} \\
		& & & \frac{1}{2}\mathbf{I} & -\mathbf{b}_{4}
	\end{bmatrix}
	\quad
	\mathbf{B}_{1} =
	\begin{bmatrix}
		\mathbf{I} & & & & \\
		& \mathbf{I} & & & \\
		& & \mathbf{I} & & \\
		& & & \mathbf{I} & \\
		& & & & 2\mathbf{b}_5
	\end{bmatrix}
\end{equation}
with
\begin{equation}
	\mathbf{X_B} =
	\begin{bmatrix}
		\mathbf{0} & \mathbf{0} & \mathbf{0} & \mathbf{0} & \mathbf{I}
	\end{bmatrix}
\end{equation}
and
\begin{equation}
	\mathbf{Y_B} =
	\begin{bmatrix}
		\mathbf{I} & \mathbf{0} & \mathbf{0} & \mathbf{0} & \mathbf{0}
	\end{bmatrix}^{\mathrm{T}} \>.
\end{equation}
That is, $\mathbf{X_B}(z\mathbf{B}_{1} - \mathbf{B}_{0})^{-1}\mathbf{Y_B} = \mathbf{b}^{-1}(z)$, $z \notin \Lambda(b)$.

Specifically, we take for $\mathbf{a}(z)$ the  nodes $\begin{bmatrix}-1 & -\frac{1}{2} & \frac{1}{2} & 1\end{bmatrix}$ and the barycentric weights $\beta = \begin{bmatrix}-\frac{2}{3} & \frac{4}{3} & -\frac{4}{3} & \frac{2}{3}\end{bmatrix}$. We suppose that
\begin{align}
	\mathbf{a}(-1) &= 
	\begin{bmatrix}
		-2 & -1 & -1 \\
		-1 & -1 & 1 \\
		0 & -1 & -1
	\end{bmatrix} \\
	\mathbf{a}(-\sfrac{1}{2}) &= 
	\begin{bmatrix}
		-0.875 & -0.5 & -1.25 \\
		-0.75 & -0.125 & 0.5 \\
		0 & -0.75 & -0.875
	\end{bmatrix}\\
	\mathbf{a}(\sfrac{1}{2}) &= 
	\begin{bmatrix}
		-1.625 & 0.5 & -0.25 \\
		-1.75 & 0.125 & -0.5 \\
		0 & -1.75 & -0.625
	\end{bmatrix}\\
	\mathbf{a}(1) &=
	\begin{bmatrix}
		-2 & 1 & 1 \\
		-3 & 1 & -1 \\
		0 & -3 & 1
	\end{bmatrix}
\end{align}
Thus $\mathbf{a}(z)$ has degree at most $3$. We choose $\mathbf{b}(z)$ of degree $3$, with
\begin{align*}
	\mathbf{b}_{0} &= 
	\begin{bmatrix}
		\phantom{-}0 & -1 & \phantom{-}0 \\
		\phantom{-}1 & -1 & -1 \\
		-1 & \phantom{-}1 & \phantom{-}0
	\end{bmatrix}\\
	\mathbf{b}_{1} &= 
	\begin{bmatrix}
		\phantom{-}0 & \phantom{-}1 & \phantom{-}0 \\
		-1 & -1 & \phantom{-}1\\
		\phantom{-}0 & -1 & -1
	\end{bmatrix}\\
	\mathbf{b}_{2} &= 
	\begin{bmatrix}
		\phantom{-}1 & -1 & \phantom{-}0 \\
		-1 & -1 & -1 \\
		\phantom{-}0 & -1 & \phantom{-}0
	\end{bmatrix}\\
	\mathbf{b}_{3} &= 
	\begin{bmatrix}
		1 & 0 & 0 \\
		0 & 1 & 0 \\
		0 & 0 & 1
	\end{bmatrix} \>.
\end{align*}
The shape of the resulting algebraic linearization for $\mathbf{h}(z) = z\mathbf{a}(z)\mathbf{b}(z) + \mathbf{c}_0$ is shown in figure \ref{fig:shape}.
\begin{figure}
	\centering
	\begin{subfigure}[b]{0.475\textwidth}
		\fbox{\includegraphics[width=\textwidth]{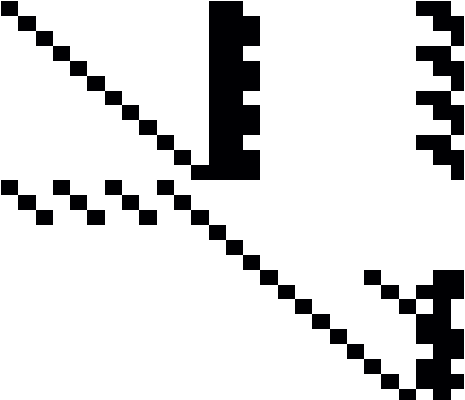}}
		\caption{Matrix structure of $\mathbf{H}$}
	\end{subfigure}
	\quad
	\begin{subfigure}[b]{0.475\textwidth}
		\fbox{\includegraphics[width=\textwidth]{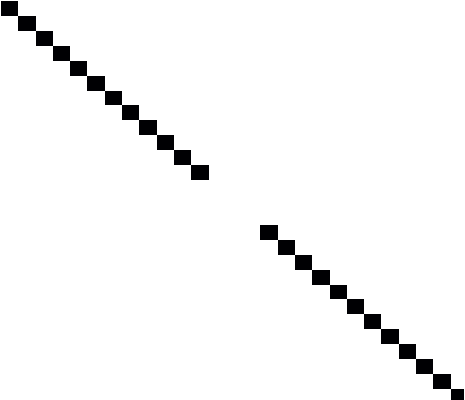}}
		\caption{Matrix structure of $\mathbf{D_H}$}
	\end{subfigure}
	\caption{Matrix structure of the companion matrix $(\mathbf{H}, \mathbf{D_H})$ of $\mathbf{h}(z) = z\mathbf{a}(z)\mathbf{b}(z) + \mathbf{c}_0$. The block of zeros in $\mathbf{D_H}$ means that there are spurious infinite eigenvalues. These are numerically harmless and can be discarded.}
	\label{fig:shape}
\end{figure}

To find the forward error of the eigenvalues, we needed a program to find the appropriate root/eigenvalue pairings. Because the number of eigenvalues and roots in this test was modest, we wrote this ``sibling finder'' program in Maple. The largest forward error of this construction is approximately $8.7 \times 10^{-15}$.

While not conclusive, these experiments show that the algebraic linearization introduced this paper can be fast and accurate when computing polynomial eigenvalues.

\begin{remark}
We learned to be careful not to have singular $\mathbf{c}_{k}$, which leads to high multiplicity zero eigenvalues of $\mathbf{h}_{k}(z)$ and thus of $\mathbf{H}_{k}$. Such high multiplicity zeros caused serious numerical artifacts. Owing to the integer nature of this family, this could perhaps be ameliorated without recourse to high precision, but we leave this for future work.
\end{remark}

\section{Concluding Remarks}

\begin{quote}
	``Almost anything will give you a strong linearization. What would be interesting would be numerical stability.'' --- Fran\c{c}oise Tisseur (private communication)
\end{quote}
There is some hope here for numerical stability of these linearizations, owing to the reduced height. Indeed, taken to extremes, a linearization of height 1 might have a characteristic equation of height exponential in the degree. This means that the polynomial evaluation condition number~\cite{corless2013graduate} will be $\mathcal{C^{N}}$ for some $C > 1$. However, the linearization resulting from recursive use of Theorem \ref{thm:4}, having height 1, will have an expected condition number $\mathcal{O}(N^{2})$~\cite{beltran2017polynomial}. Here $N$ is the dimension of the matrix. This means that the algorithm implied by the use of our linearizations can be (for some examples) exponentially more numerically stable.

However, not every matrix polynomial has a naturally recursive formulation. Preliminary experiments on reverse-engineering such formulations are promising and we will report on these developments later.

We have no theorems that suggest a lower-height matrix will have  better-conditioned eigenvalues, only an expectation that is perhaps naive. This, too, will be reported on at a later date. Of course by ``height'' we mean scaled height, which needs a careful formulation; obviously $s\mathbf{A}$ has eigenvalues $s\lambda_{k}$ if $\mathbf{A}$ has eigenvalues $\lambda_{k}$, and the same eigenvectors (and thus eigenvalue condition numbers are unchanged by the scaling). Perhaps a better numerical representation of scaled height's sensitivity would be, say,
\begin{equation}
	t = \min_{a_{ij}\neq 0}\frac{|a_{ij}|}{\mathrm{Height}(\mathbf{A})} \>.
\end{equation}
The smaller this number is, the more sensitive one might expect the eigenvalues to  be. Again, this has yet to be explored.

\section*{Acknowledgments}
We acknowledge the support of the Ontario Graduate Institution, The National Science \& Engineering Research Council of Canada, the University of Alcal\'{a}, the Rotman Institute of Philosophy, the Ontario Research Centre of Computer Algebra, and Western University. Part of this work was developed while R.~M.~Corless was visiting the University of Alcal\'{a}, in the frame of the project Giner de los Rios. L.~Gonzalez-Vega, J.~R.~Sendra and J.~Sendra are partially supported by the Spanish Ministerio de Econom\'{i}a y Competitividad under the Project MTM2014-54141-P.

\bibliography{references}

\begin{thebibliography}{10}
\expandafter\ifx\csname url\endcsname\relax
  \def\url#1{\texttt{#1}}\fi
\expandafter\ifx\csname urlprefix\endcsname\relax\def\urlprefix{URL }\fi
\expandafter\ifx\csname href\endcsname\relax
  \def\href#1#2{#2} \def\path#1{#1}\fi

\bibitem{dopico2018block}
F.~M. Dopico, J.~P{\'e}rez, P.~Van~Dooren, Block minimal bases $\ell
  $-ifications of matrix polynomials, arXiv preprint arXiv:1803.06306.

\bibitem{gohberg2009matrix}
I.~Gohberg, P.~Lancaster, L.~Rodman, Matrix polynomials, SIAM, 2009.

\bibitem{dennis2015princeton}
M.~R. Dennis, P.~Glendinning, P.~A. Martin, F.~Santosa, J.~Tanner, The
  Princeton companion to applied mathematics, Princeton University Press, 2015.

\bibitem{rahman2002analytic}
Q.~I. Rahman, G.~Schmeisser, Analytic theory of polynomials, no.~26, Oxford
  University Press, 2002.

\bibitem{mackey2006vector}
D.~S. Mackey, N.~Mackey, C.~Mehl, V.~Mehrmann, Vector spaces of linearizations
  for matrix polynomials, SIAM Journal on Matrix Analysis and Applications
  28~(4) (2006) 971--1004.

\bibitem{mackey2006structured}
D.~S. Mackey, N.~Mackey, C.~Mehl, V.~Mehrmann, Structured polynomial eigenvalue
  problems: Good vibrations from good linearizations, SIAM Journal on Matrix
  Analysis and Applications 28~(4) (2006) 1029--1051.

\bibitem{Chan2016}
E.~Y.~S. Chan, \href{http://ir.lib.uwo.ca/etd/4028}{A comparison of solution
  methods for {M}andelbrot-like polynomials}, Master's thesis, The University
  of Western Ontario (2016).
\newline\urlprefix\url{http://ir.lib.uwo.ca/etd/4028}

\bibitem{corless2013largest}
R.~M. Corless, P.~W. Lawrence, The largest roots of the {M}andelbrot
  polynomials, in: Computational and Analytical Mathematics, Springer, 2013,
  pp. 305--324.

\bibitem{Chan2017}
E.~Y.~S. Chan, R.~M. Corless, A new kind of companion matrix, Electronic
  Journal of Linear Algebra 32 (2017) 335--342.

\bibitem{hogben2006handbook}
L.~Hogben, Handbook of linear algebra, CRC Press, 2006.

\bibitem{char2012first}
B.~W. Char, K.~O. Geddes, G.~H. Gonnet, B.~L. Leong, M.~B. Monagan, S.~M. Watt,
  First leaves: a tutorial introduction to {M}aple {V}, Springer Science \&
  Business Media, 2012.

\bibitem{amiraslani2008linearization}
A.~Amiraslani, R.~M. Corless, P.~Lancaster, Linearization of matrix polynomials
  expressed in polynomial bases, IMA Journal of Numerical Analysis 29~(1)
  (2008) 141--157.

\bibitem{van2015linearization}
R.~Van~Beeumen, W.~Michiels, K.~Meerbergen, Linearization of {L}agrange and
  {H}ermite interpolating matrix polynomials, IMA Journal of Numerical Analysis
  35~(2) (2015) 909--930.

\bibitem{corless2004generalized}
R.~M. Corless, Generalized companion matrices in the lagrange basis,
  Proceedings of EACA (2004) 317--322.

\bibitem{lawrence2014stability}
P.~W. Lawrence, R.~M. Corless, Stability of rootfinding for barycentric
  {L}agrange interpolants, Numerical Algorithms 65~(3) (2014) 447--464.

\bibitem{good1961colleague}
I.~Good, The colleague matrix, a {C}hebyshev analogue of the companion matrix,
  The Quarterly Journal of Mathematics 12~(1) (1961) 61--68.

\bibitem{specht1956lage}
W.~Specht, Die lage der nullstellen eines polynoms, Mathematische Nachrichten
  15~(5-6) (1956) 353--374.

\bibitem{corless2013graduate}
R.~M. Corless, N.~Fillion, A graduate introduction to numerical methods, AMC 10
  (2013) 12.

\bibitem{beltran2017polynomial}
C.~Beltr{\'a}n, D.~Armentano, The polynomial eigenvalue problem is well
  conditioned for random inputs, arXiv preprint arXiv:1706.06025.

\end{thebibliography}
\end{document}